\documentclass[default]{sn-jnl}

\usepackage{graphicx}%
\usepackage{multirow}%
\usepackage{amsmath,amssymb,amsfonts}%
\usepackage{amsthm}%
\usepackage{mathrsfs}%
\usepackage[title]{appendix}%
\usepackage{xcolor}%
\usepackage{textcomp}%
\usepackage{manyfoot}%
\usepackage{booktabs}%
\usepackage{algorithm}%
\usepackage{algorithmicx}%
\usepackage{algpseudocode}%
\usepackage{listings}%
\usepackage{accents}
\usepackage{comment}
\usepackage{multicol}
\usepackage{bbm}

\usepackage{wrapfig} 
\usepackage{framed} 
\usepackage[justification=centering]{caption}  
\usepackage{pgfplots}
\pgfplotsset{compat=1.15}
\usetikzlibrary{arrows}
\usepackage{enumerate} 
\usepackage{enumitem}
\usepackage{caption}  
\usepackage{hyperref} 
\usepackage{geometry}
\numberwithin{equation}{section}
\newtheorem{Lemma}{Lemma}[section]
\newtheorem{lemma}{Lemma}[section]
\newtheorem{Proposition}[Lemma]{Proposition}

\newtheorem{Theorem}[Lemma]{Theorem}

\newtheorem{Remark}[Lemma]{Remark}

\newtheorem{Definition}[Lemma]{Definition}

\newtheorem{remark}[Lemma]{Remark}

\newtheorem{Claim}[Lemma]{Claim}
\newcommand{\bfa} {{\mathbf a}}

\newcommand{\bfp} {{\mathbf p}}
\newcommand{\bfq} {{\mathbf q}}

\newcommand{\bfu} {{\mathbf u}}

\newcommand{\bfx} {{\mathbf x}}

\newcommand{\bfM} {{\mathbf M}}

\newcommand{\bfT} {{\mathbf T}}

\newcommand{\Acal}   {\mathcal{A}}
\newcommand{\Bcal}   {\mathcal{B}}
\newcommand{\Ccal}   {\mathcal{C}}
\newcommand{\Dcal}   {\mathcal{D}}

\newcommand{\Fcal}   {\mathcal{F}}

\newcommand{\Mcal}   {\mathcal{M}}

\newcommand{\Scal}   {\mathcal{S}}
\newcommand{\Tcal}   {\mathcal{T}}

\newcommand{\e}{{\mathrm e}}

\newcommand{\R}{\mathbb{R}}

\newcommand{\N}{\mathbb{N}}

\newcommand{\Qbb}{\mathbb{Q}}

\newcommand{\E}{\mathbb{E}}

\newcommand{\prob}{\mathbb{P}}

\newcommand{\emp}{\varnothing}

\newcommand{\eqn}[1]{\begin{equation} #1 \end{equation}}
\newcommand{\eqan}[1]{\begin{align} #1 \end{align}}

\newcommand{\one}{\mathbbm{1}}

\newcommand{\PA}{\mathrm{PA}}

\newcommand{\RPPT}{\mathrm{RPPT}}
\newcommand{\PPT}{\mathrm{PPT}}
\newcommand{\sss}{\scriptscriptstyle}

\newcommand{\din}{d^{\sss(\mathrm{in})}}

\newcommand{\old}{{\scriptscriptstyle {\sf O}}}
\newcommand{\Old}{{\scriptstyle {\sf O}}}
\newcommand{\young}{{\scriptscriptstyle {\sf Y}}}
\newcommand{\Young}{{\scriptstyle {\sf Y}}}

\newcommand{\btr}{$\triangleright$}
\newcommand{\vol}{\mathrm{vol}}
\newcommand{\cut}{\mathrm{C}}
\newcommand{\nn}{\nonumber}
\newcommand{\vep}{\varepsilon}

\newcommand{\FC}[1]{#1}
\newcommand{\arxiversion}[1]{}

\definecolor{darkred}{rgb}{1,0,0}
\definecolor{darkgreen}{rgb}{0,0.6,0}
\definecolor{darkblue}{rgb}{0,0,1}
\definecolor{darkagenta}{rgb}{0.6,0,0.6}
\definecolor{wrwrwr}{rgb}{0.3803921568627451,0.3803921568627451,0.3803921568627451}
\definecolor{rvwvcq}{rgb}{0.08235294117647059,0.396078431372549,0.7529411764705882}

\theoremstyle{thmstyleone}%
\theoremstyle{thmstyletwo}%

\theoremstyle{thmstylethree}%

\begin{document}

\title[Percolation on preferential attachment models]{Percolation on preferential attachment models}


\author[1]{\fnm{Rajat} \sur{Hazra}}\email{r.s.hazra@math.leidenuniv.nl}

\author[2]{\fnm{Remco} \sur{van der Hofstad}}\email{r.w.v.d.hofstad@TUE.nl}

\author[2]{\fnm{Rounak} \sur{Ray}}\email{r.ray@tue.nl}

\affil[1]{\orgdiv{Department of Mathematics}, \orgname{University of Leiden}, \orgaddress{\country{The Netherlands}}}

\affil[2]{\orgdiv{Department of Mathematics and Computer Science}, \orgname{Eindhoven University of Technology}, \orgaddress{\country{The Netherlands}}}



\abstract{We study the percolation phase transition on preferential attachment models, in which vertices enter with $m$ edges and attach proportionally to their degree plus $\delta$. We identify the critical percolation threshold as $$\pi_c=\frac{\delta}{2\big(m(m+\delta)+\sqrt{m(m-1)(m+\delta)(m+1+\delta)}\big)}$$ for $\delta$ positive and $\pi_c=0$ for non-positive values of $\delta$. Therefore the giant component is robust for $\delta\in(-m,0]$, while it is not for $\delta>0$.
	
Our proof for the critical percolation threshold consists of three main steps. 
First, we show that preferential attachment graphs are large-set expanders, enabling us to verify the conditions outlined by Alimohammadi, Borgs, and Saberi (2023). Within their conditions, the proportion of vertices in the largest connected component in a sequence converges to the survival probability of percolation on the local limit. In particular, the critical percolation threshold for both the graph and its local limit are identical.
Second, we identify $1/\pi_c$ as the spectral radius of the mean offspring operator of the P\'olya point tree, the local limit of preferential attachment models.
Lastly, we prove that the critical percolation threshold for the P\'olya point tree is the inverse of the spectral radius of the mean offspring operator. For positive $\delta$, we use sub-martingales to prove sub-criticality and apply spine decomposition theory to demonstrate super-criticality, completing the third step of the proof. For $\delta\leq 0$ and any $\pi>0$ instead, we prove that the percolated P\'olya point tree dominates a supercritical branching process, proving that the critical percolation threshold equals $0$.}




\maketitle

\section{Introduction and main results}
\subsection{Introduction}
Bond percolation is a fundamental and elementary process in statistical physics and network science that investigates the connectivity of a network under random attack. In bond percolation, each edge of a given graph is independently retained with probability $\pi$ and removed with probability $1-\pi$. Random graph models typically undergo a phase transition in their connectivity structure depending on the percolation probability $\pi$. Specifically, there exists a critical percolation threshold $\pi_c \in [0,1]$ such that for $\pi > \pi_c$, the proportion of vertices in the largest connected component is strictly positive, whereas for $\pi < \pi_c$, this proportion converges to zero. In other words,
\eqn{\label{def:critical-percolation}
\frac{|\Ccal_1^{(n)}(\pi)|}{n} \to c(\pi)~,}
where $c(\pi)>0$ for $\pi>\pi_c$ and $0$ when $\pi<\pi_c$ and $\Ccal_1^{(n)}(\pi)$ is the largest connected component of the percolated graph of size $n$.
$\pi_c$ is known as the \emph{critical percolation threshold} of the random graph.

Percolation on random graphs has an extensive and diverse literature, and we refer the reader to \cite{BJR07,ER60,vdH1,JL09,JPRR18,MR95} and the references therein. Local convergence techniques have emerged as valuable tools for studying the size and uniqueness of the giant component. However, these properties do not immediately follow from local convergence.

In \cite[Conjecture~1.2]{BNP11}, Schramm conjectured that the critical percolation threshold exhibits locality for a sequence of vertex-transitive infinite graphs converging locally to a limiting graph. Various attempts have been made to prove this conjecture for different subclasses of infinite graphs in \cite{BNP11,S21,ABS22,KLS20,vdH23}. In particular, van der Hofstad in \cite{vdH23} provided a simple necessary and sufficient condition for the conjecture to hold, while \cite{BNP11,S21,KLS20} proved it for bounded-degree expanders. Recently, Easo and Hutchcroft proved the Schramm conjecture in \cite{EH23}. Furthermore, \cite{ABS22} established the conjecture for large-set expander graphs with bounded average degree and extended the result to include directed percolation as well. This result extends the Schramm's conjecture beyond the scope of vertex-transitive graphs. The local limit of the preferential attachment model was initially explored in \cite{BergerBorgs}. Subsequently, \cite{RRR22} demonstrated the universality of this local limit, encompassing a broad spectrum of preferential attachment models.

Although percolation on static random graphs has been extensively studied, percolation on preferential attachment models is relatively less explored. In \cite{EMO21}, the authors studied the proportion of vertices in the largest connected component of the Bernoulli preferential attachment model defined in \cite{DM13}, where the attachment rule is a function of the number of incoming edges of a vertex. In \cite{ABS22}, the preferential attachment model for the independent edge attachment rule and $\delta=0$ was studied. It was shown that the model is a large-set expander with bounded average degree, and its critical percolation threshold is $0$. Additionally, it was demonstrated that the phase transition is of infinite order in that case as proved in \cite{EMO21} for the model in \cite{DM13}. The fact that $\delta=0$ plays a crucial role in this proof, which mainly relies on functional analysis. In this paper, we extend the work in \cite{ABS22} to several other preferential attachment models with affine attachment functions and a general $\delta$ parameter.
\subsection{The models}
In this paper, we focus on sequential preferential attachment models, specifically models (a), (b) and (d) as defined in \cite{vdH1}, excluding their tree cases. We fix $m\in\N\setminus{\{1\}}$ and $\delta>-m$. In these models, every new vertex is introduced to the existing graph with $m$ edges incident to it. The models are defined by their edge-connection probabilities.

We start from any finite graph with $2$ vertices and finitely many connections between them such that at least one of the initial vertices have degree at most $m$. Let $a_1$ and $a_2$ denote the degrees of the vertices $1$ and $2$, respectively. Without loss of generality, consider $a_2\leq m$. We also allow for self-loops in the initial graph.
\paragraph{Model (a).} For each new vertex $v$ joining the graph and $j=1,2,\ldots,m$, the attachment probabilities are given by
\eqn{\label{def:model:a}
	\prob\Big( v\overset{j}{\rightsquigarrow} u\mid\PA^{(a)}_{v,j-1}(m,\delta) \Big)= \begin{cases}
		\frac{d_u(v,j-1)+\delta}{c_{v,j}^{(a)}}\hspace{1.15cm}\text{for}~u<v,\\
		\frac{d_v(v,j-1)+{j\delta}/{m}}{c_{v,j}^{(a)}}\hspace{0.65cm}\text{for}~u=v,
	\end{cases}
}
where $v\overset{j}{\rightsquigarrow} u$ denotes that vertex $v$ connects to $u$ with its $j$-th edge, $\PA^{(a)}_{v,j}({m},\delta)$ denotes the graph on $v$ vertices, \FC{with} the $v$-th vertex \FC{having} $j$ out-edges, and $d_u(v,j)$ denotes the degree of vertex $u$ in $\PA^{(a)}_{v,j}({m},\delta)$. We identify $\PA^{(a)}_{v+1,0}({m},\delta)$ with $\PA^{(a)}_{v,m}({m},\delta)$. The normalizing constant $c_{v,j}^{(a)}$ in \eqref{def:model:a} equals
\eqn{\label{eq:normali}
	c_{v,j}^{(a)} := a_{[2]}+2\delta+(2m+\delta)(v-3)+2(j-1)+1+\frac{j\delta}{m} \, ,
}
where $a_{[2]}=a_1+a_2$. We denote the above model by $\PA^{(a)}_{v}({m},\delta)$, which is equivalent to $\PA_v^{(m,\delta)}(a)$ as defined in \cite{vdH1}.
\paragraph{Model (b).}
For every new vertex $v$ joining the graph and $j=1,2,\ldots,m$, the attachment probabilities are given by
\eqn{\label{def:model:b}
	\prob\Big( v\overset{j}{\rightsquigarrow} u\mid\PA^{(b)}_{v,j-1}(m,\delta) \Big)= \begin{cases}
		\frac{d_u(v,j-1)+\delta}{c_{v,j}^{(b)}}\hspace{1.8cm}\text{for}~u<v,\\
		\frac{d_v(v,j-1)+{(j-1)\delta}/{m}}{c_{v,j}^{(b)}}\hspace{0.65cm}\text{for}~u=v,
	\end{cases}
}
with all notation as before, but now for model (b), for which
\eqn{\label{eq:normalb}
	c_{v,j}^{(b)} := a_{[2]}+2\delta+(2m+\delta)(v-3)+2(j-1)+\frac{(j-1)\delta}{m} \, .
}
We denote the above model as $\PA^{(b)}_{v}({m},\delta)$, which is equivalent to $\PA_v^{(m,\delta)}(b)$ as defined in \cite{vdH1}.

\begin{remark}[Difference between models (a) and (b)]
	\rm{Models (a) and (b) are different in that the first edge from every new vertex can create a self-loop in model (a) but not in model (b). Note that the edge probabilities are different for all $j$.}\hfill$\blacksquare$
\end{remark}

\paragraph{Model (d).}
For every new vertex $v$ joining the graph and $j=1,2,\ldots,m$, the attachment probabilities are given by
\eqn{\label{def:model:d}
	\prob\Big( v\overset{j}{\rightsquigarrow} u\mid\PA^{(d)}_{v,j-1}(m,\delta) \Big)= \frac{d_u(v,j-1)+\delta}{c_{v,j}^{(d)}}\hspace{1.5cm}\text{for}~u<v~,
}
with all notation as before, but now for model (d), for which
\eqn{\label{eq:normald}
	c_{v,j}^{(d)} := a_{[2]}+2\delta+(2m+\delta)(v-3)+(j-1)\, .
}
We denote the above model as $\PA^{(d)}_{v}({m},\delta)$, which is essentially equivalent to $\PA_v^{(m,\delta)}(d)$ as defined in \cite{vdH1}.

\subsection{Main results}\label{subsec:main-theorem}
Fix $m\in\N\setminus\{1\}$ and $\delta>-m$. Throughout this paper, $(G_n)_{n\geq2}$ will denote $\big(\PA^{ (s)}_n(m,\delta)\big)_{n\geq 2}$ with $s\in\{a,b,d\}$. We now state the main result of this article:
\begin{Theorem}[Critical Percolation Threshold for $\PA$ models]\label{thm:main:theorem}
	For any $\pi\in[0,1]$, let $\Ccal_i^{(n)}(\pi)$ denote the $i$-th largest connected component of the percolated $G_n$.
	Then
	\eqn{\label{eq:main:theorem}
	\frac{|\Ccal_1^{(n)}(\pi)|}{n} \overset{\prob}{\to}\zeta(\pi)
	\quad\text{and}\quad\frac{|\Ccal_2^{(n)}(\pi)|}{n} \overset{\prob}{\to}0~,}
	where $\overset{\prob}{\to}$ denotes convergence in probability with respect to both the random graph and percolation. Furthermore, define 
	\eqn{\label{:def:pi-c}
	\pi_c=\begin{cases}
		0, &\text{for}\quad \delta\in(-m,0]\\
		\frac{\delta}{2\big(m(m+\delta)+\sqrt{m(m-1)(m+\delta)(m+1+\delta)}\big)}, &\text{for}\quad\delta>0
	\end{cases}.}
	Then
	$\zeta(\pi)>0$ for $\pi>\pi_c$, whereas $\zeta(\pi)=0$ for $\pi\leq \pi_c$.
\end{Theorem}
\paragraph{\textbf{Observations and discussions}.} We make a few remarks about the above result:
\begin{enumerate}
	\item Our proof critically relies on the local limit of preferential attachment models. In \cite{RRR22} (see also \cite{BergerBorgs}), we observed that a large class of preferential attachment models with affine edge attachment rules have the same local limit. While we have specifically proven Theorem~\ref{thm:main:theorem} for models (a), (b) and (d), it is reasonable to expect that the result holds for other versions of preferential attachment models as well. 
	\item The authors in \cite{ABS22} have demonstrated that the result holds true for the model where edge-connection probabilities of the new vertex are independent, and $\delta=0$. \cite{DM13} provided the critical percolation threshold for Bernoulli preferential attachment model, whereas \cite{EMO21} proved that the order of phase transition for the same model is infinite.
	\item 
	We use the condition $m\geq 2$ crucially to prove Theorem~\ref{thm:main:theorem}. Our proof fails because trees are not expanders. In \cite{RTV07}, preferential attachment models were investigated from the perspective of continuous-time embeddings, which turns the model into a continuous-time branching process. Percolation on a continuous-time branching process is still a continuous-time branching process, which allows one to investigate also the tree setting arising when $m=1$. We abstain from performing the explicit computations for this case.
	\item The parameter $1/\pi_c$ can be interpreted as the spectral radius of the mean-offspring operator of the local limit of $\PA$s. It can be expected to play a crucial role in other structural properties of $\PA$s. For example, \cite{vdH2} proves a lower bound on typical distances for $\PA$s with $\delta>0$ of order $\log n/\log (1/\pi_c)$. See also \cite{DHH10} for related distance results in $\PA$s.
\end{enumerate}

\section{Overview of proof of Theorem~\ref{thm:main:theorem}}\label{sec:proof-strategy} Under certain conditions on the random graph sequence, the giant component of the graph after percolating with probability $\pi$ becomes a local property. This means that the limiting behavior of the largest connected component in the percolated graphs can be determined from the percolated local limit of the graph sequence. We utilize the conditions provided by Alimohammadi, Borgs and Saberi in \cite{ABS22}. Before continuing with the overview of the proof we describe local convergence in the space of marked rooted graphs, as well as the P\'olya point tree as the local limit of preferential attachment models.

\subsection{The space of rooted vertex-marked graphs and marked local convergence}
Local convergence of rooted graphs was first introduced by Benjamini and Schramm in \cite{benjamini} {and Aldous and Steele in \cite{Aldous2004}.} {We now give a brief introduction from \cite{vdH2}.} 

A graph $G=(V(G), E(G))$ ({possibly} infinite) is called \textit{locally finite} if every vertex $v\in V(G)$ has finite degree (not necessarily uniformly bounded). A pair $(G,o)$ is called a {{\em rooted graph}, where $G$ is} rooted at $o\in V(G)$. For any two vertices $u,v\in V(G),~d_G(u,v)$ is defined as the length of the shortest path from $u$ to $v$ in $G$, {i.e., the minimal number of edges needed to connect $u$ and $v$}. We let $d_G(u,v)=\infty$ when $u$ and $v$ are not connected. 

For a rooted graph $(G,o)$, the {\em $r$-neighbourhood} of the root $o$, {denoted by $B_r^{\sss(G)}(o)$, is defined as the graph rooted at $o$ with vertex and edge sets given by}
\eqn{\label{eq:def:nbhd:root}
	\begin{split}
		V\big( B_r^{\sss(G)}(o) \big) =& \{ v\in V(G)\colon d_G(o,v)\leq r \},\quad\text{and}\\
		E\big( B_r^{\sss(G)}(o) \big) =& \left\{ \{ u,v \}\in E(G)\colon u,v\in V\big( B_r^{\sss(G)}(o) \big) \right\}\,.
\end{split}}
Let $(G_1,o_1)$ and $(G_2,o_2)$ be two rooted locally finite graphs with $G_1=(V(G_1),E(G_1))$ and $G_2=(V(G_2),E(G_2))$. Then we say that $(G_1,o_1)$ is {\em{rooted isomorphic}} to $(G_2,o_2)$, {which we denote as $(G_1,o_1)\simeq (G_2,o_2)$}, {when} there exists a graph isomorphism between $G_1$ and $G_2$ that maps $o_1$ to $o_2$, i.e., {when} there exists a bijection $\phi\colon V(G_1)\mapsto V(G_2)$ such that 
\eqn{\label{def:rooted:isomorphism}
	\begin{split}
		&\phi(o_1)=o_2,\qquad\text{and} \qquad\{u,v\}\in E(G_1) \iff \{\phi(u),\phi(v)\}\in E(G_2).
\end{split}} 

{\em Marks} are generally images of an injective function $\mathcal{M}$ acting on the vertices of a graph $G,$ as well as on the edges. 
{These} marks take values in a complete separable metric space $(\Xi,d_{\Xi})$. Rooted graphs with only vertices having marks are called \textit{vertex-marked} rooted graphs, {and are denoted} by $(G,o,\mathcal{M}(G))= (V(G),E(G),o,\mathcal{M}(V(G)))$. We say that two rooted vertex-marked graphs $(G_1,o_1,\mathcal{M}_1(G_1))$ and $(G_2,o_2,\mathcal{M}_2(G_2))$ are vertex-marked-isomorphic, {which we denote as $(G_1,o_1,\mathcal{M}_1(G_1))\overset{\star}{\simeq}(G_2,o_2,\mathcal{M}_2(G_2))$,} {when} there exists a bijective function $\phi\colon V(G_1)\mapsto V(G_2)$ such that
\begin{itemize}
	\item[$\rhd$] $\phi(o_1)=o_2$;
	\item[$\rhd$] $\{v_1,v_2\}\in E(G_1)$ if and only if $\{\phi(v_1),\phi(v_2)\}\in E(G_2)$;
	\item[$\rhd$] $\mathcal{M}_1(v)=\mathcal{M}_2(\phi(v))$ for all $v\in V(G_1)$.
\end{itemize}

{Let} $\mathcal{G}_\star$ be the vertex-marked isomorphism invariant class of rooted graphs. Similarly as in {\cite[Definition 2.11]{vdH2}}, one can define a metric $d_{\mathcal{G}_\star}$ as
\eqn{\label{for:def:distance:1}
	d_{\mathcal{G}_\star}\left( (G_1,o_1,\mathcal{M}_1(G_1)),(G_2,o_2,\mathcal{M}_2(G_2)) \right) = \frac{1}{1+R^\star}\, ,}
where
\eqan{\label{for:def:distance:2}
	R^\star = \sup\{ r\geq& 0: B_r^{(G_1)}(o_1)\simeq B_r^{(G_2)}(o_2),\mbox{and there exists a rooted isomorphism $\phi$ from $V(G_1)$}\nn\\
	&\mbox{to $V(G_2)$ such that }d_{\Xi}(\mathcal{M}_1(u),\mathcal{M}_2(\phi(u)))\leq \frac{1}{r},~\forall u\in V(B_r^{(G_1)}(o_1))\},}
which makes $\left(\mathcal{G}_\star,d_{\mathcal{G}_\star}\right)$ a Polish space.
We {next} define the notion of marked local convergence of vertex-marked random graphs on this space. {\cite[Theorem~2.14]{vdH2} describes various} notions of local convergence.
For the definition of local convergence, we consider that $ (G_n)_{n\geq 1}$ is a sequence of (possibly random) vertex-marked graphs with $G_n=\left( V(G_n),E(G_n),\mathcal{M}_n(V(G_n)) \right)$ that are finite (almost surely). Conditionally on $G_n,$ let $o_{n}$ be a randomly chosen vertex from $V(G_n)$. Note that $((G_n,o_n))_{n\geq 1}$ is a sequence of random variables defined on $\mathcal{G}_\star$. Then {vertex-marked local convergence is defined as follows:}
\begin{Definition}[Vertex marked local  convergence]\label{def:vertex:marked:local:convergence}
			The sequence $(G_n)_{n\geq 1}$ is said to converge {{{vertex-marked locally}}} in probability to a random element $(G,o,\mathcal{M}(V(G)))\in \mathcal{G}_\star$ having probability law $\mu_\star$, when for every $r>0$ and for every $(H_\star,\mathcal{M}_{H_\star}(H_\star))\in\mathcal{G}_\star$, as $n\to\infty$,
			\eqan{\label{eq:def:mark:local:convergence}
				&\frac{1}{n}\sum\limits_{\omega\in[n]}\one_{\left\{ B_r^{\sss(G_n)}({\omega})\simeq H_\star,d_{\mathcal{G}_\star}\big( (B_r^{\sss(G_n)}({\omega}),\omega,\mathcal{M}(V(B_r^{\sss(G_n)}({\omega}))),(H_\star,\mathcal{M}_{H_\star}(H_\star)) \big)\leq \frac{1}{r} \right\}}\\
				&\hspace{0.5cm}\overset{\prob}{\to} \mu_\star\left( B_r^{\sss(G)}({o})\simeq H_\star,~d_{\mathcal{G}_\star}\big( (B_r^{\sss(G)}({o}),o,\mathcal{M}(V(B_r^{\sss(G)}({o}))),(H_\star,\mathcal{M}_{H_\star}(H_\star)) \big)\leq \frac{1}{r} \right).\nn}
	\hfill$\blacksquare$
\end{Definition}
\subsection{The P\'{o}lya point tree}\label{sec:RPPT}
Berger, Borgs, Chayes and Saberi proved in \cite{BergerBorgs}, that the P\'olya point tree is the local limit of preferential attachment model. In this section, {we define the P\'olya point tree ($\PPT$) that will act as the vertex-marked local limit of our preferential attachment graphs. In \cite{RRR22}, the authors have proved that the random P\'olya point tree ($\RPPT$) is the vertex-marked local limit of a wide class of preferential attachment models with i.i.d.\ out-degrees. Further, the P\'olya point tree could be shown to be a special case of the random P\'olya point tree. We adapt the definition of $\PPT$ from \cite{BergerBorgs} and \cite{RRR22}. We start by defining the vertex set of this PPT:}

\begin{Definition}[Ulam-Harris set and its ordering]\label{def:ulam}
	{\rm Let $\N_0=\N\cup\{0\}$. {The {\em Ulam-Harris set}} is
		\begin{equation*}
			\mathcal{N}=\bigcup\limits_{n\in\N_0}\N^n.
		\end{equation*}
		For $x = x_1\cdots x_n\in\N^n$ and $k\in\N$, we denote {the element $x_1\cdots x_nk$ by $xk\in\N^{n+1}$.} The~{\em root} of the Ulam-Harris set is denoted by $\emp\in\N^0$.
		
		For any $x\in \mathcal{U}$, we say that $x$ has length $n$ if $x\in\N^n$. {The lexicographic \textit{ordering}} between the elements of the Ulam-Harris set {is as follows:}
		\begin{itemize}
			\item[(a)] for any two elements $x,y\in\mathcal{U}$, {$x>y$ when the length of $x$ is more than that of $y$;}
			\item[(b)] if $x,y\in \N^n$ for some $n$, then $x>y$ if there exists $i\leq n,$ such that $x_j=~y_j~\forall j<i$ and $x_i>~y_i$.\hfill$\blacksquare$
	\end{itemize}}
\end{Definition}

We use the elements of the Ulam-Harris set to identify nodes in a rooted tree, since the notation in Definition~\ref{def:ulam} allows us to denote the relationships between children and parents, where for $x\in\mathcal{U}$, we denote the $k$-th child of $x$ by the element $xk$. 
\smallskip

\paragraph{\bf P\'olya Point Tree ($\PPT$).}
The $\PPT(m,\delta)$ is an {\em infinite multi-type rooted random tree}, where $m$ and $\delta>-m$ are the parameters of preferential attachment models. {It is a multi-type branching process, with a mixed continuous and discrete type space. We now describe its properties one by one.}

\paragraph{ \bf Descriptions of the distributions and parameters used.}
\begin{enumerate}
	\item[{\btr}] {Define} $\chi = \frac{m+\delta}{2m+\delta}$.
	\smallskip
	\item[{\btr}] {Let} $\Gamma_{\sf{in}}(m)$ denote a Gamma distribution with parameters $m+\delta$ and $1$.
	\smallskip
	\item[{\btr}] {Let} $\Gamma_{\sf{in}}^\prime(m)$ denote the size-biased distribution of $\Gamma_{\sf{in}}(m)$, which is also a Gamma distribution with parameters $m+\delta+1$ and $1$.
\end{enumerate}
\smallskip

\paragraph{\bf Feature of the vertices of $\PPT$.}
Below, to avoid confusion, we use `node' for a vertex in the PPT and `vertex' for a vertex in the PAM. We now discuss the properties of the nodes in $\PPT(m,\delta)$. Every node except the root in the $\PPT$ has
\begin{enumerate}
	\item[{\btr}] a {\em label} $\omega$ in the Ulam-Harris set $\mathcal{N}$ (recall Definition~\ref{def:ulam});
	\smallskip
	\item[{\btr}] an {\em age} $A_{\omega}\in[0,1]$;
	\smallskip
	\item[{\btr}] a positive number $\Gamma_{\omega}$ called its {\em strength};
	\smallskip
	\item[{\btr}] a {label in $\{{\Old},{\Young}\}$} depending on the age of the {node} and its parent, {with $\Young$ denoting that the node is younger than its parent and $\Old$ denoting that the node is older than its parent.}
\end{enumerate}
Based on its type being $\Old$ or $\Young$, every {node} $\omega$ has a number $m_{-}(\omega)$ associated to it. If $\omega$ has type $\Old$, then $m_{-}(\omega)=m$, and $\Gamma_{\omega}$ is distributed as $\Gamma_{\sf{in}}^\prime( m )$,
while if $\omega$ has type $\Young$, then $m_{-}(\omega)=m-1$, and given $m_{-}(\omega), \Gamma_{\omega}$ is distributed as $\Gamma_{\sf{in}}( m )$.
\smallskip
\begingroup
\allowdisplaybreaks
\paragraph{\bf Construction of the PPT.} {We next use the above definitions to construct the PPT using an {\em exploration process}.}
The root is special in the tree. It has label $\emp$ and its age $A_\emp$ is an independent uniform random variable in $[0,1]$. The root $\emp$ has no label in $\{\Old,\Young\}$, and we let $m_{-}(\emp)=m$.
Then the {children of the root in the} P\'{o}lya point tree {are} constructed as follows:
\endgroup	
\begin{enumerate}
	\item Sample $U_1,\ldots,U_{m_{-}(\emp)}$ uniform random variables on $[0,1]$, independent of the rest;  
	\item To nodes $\emp1,\ldots, \emp m_{-}(\emp),$ assign the ages $U_1^{1/\chi} A_\emp,\ldots,U_{m_{-}(\emp)}^{1/\chi} A_\emp$ and type $\Old$;
	\item Assign ages $A_{\emp(m_{-}(\emp)+1)},\ldots, A_{\emp(m_{-}(\emp)+\din_\emp)}$ to nodes $\emp(m_{-}(\emp)+1),\ldots, \emp(m_{-}(\emp)+\din_\emp)$. These ages are the \FC{occurrence times} given by a conditionally independent Poisson point process on $[A_\emp,1]$ defined by the random intensity
	\eqn{\label{def:rho_emp}
		\rho_{\emp}(x) = {(1-\chi)}{\Gamma_\emp}\frac{x^{-\chi}}{A_\emp^{1-\chi}},}
	and $\din_\emp$ being the total number of points of this process. Assign type $\Young$ to them;
	\item Draw an edge between $\emp$ and each of $\emp 1,\ldots, \emp(m_{-}(\emp)+\din_\emp)$; 
	\item Label $\emp$ as explored and nodes $\emp1,\ldots, \emp(m_{-}(\emp)+\din_\emp)$ as unexplored.
\end{enumerate}	
Then, recursively over the elements in the set of unexplored nodes, we perform the following breadth-first exploration:
\begin{enumerate}
	\item Let $\omega$ denote the smallest currently unexplored node in the Ulam-Harris ordering;
	\item Sample $m_{-}(\omega)$ i.i.d.\ random variables $U_{\omega1},\ldots,U_{\omega m_{-}(\omega)}$ independently from all the previous steps and from each other, uniformly on $[0,1]$. To nodes $\omega 1,\ldots,\omega m_{-}(\omega)$ assign the ages $U_{\omega1}^{1/\chi}A_\omega,\ldots,U_{\omega m_{-}(\omega)}^{\FC{1/\chi}}A_\omega$ and type $\Old$, and set them unexplored;
	\item Let $A_{\omega(m_{-}(\omega)+1)},\ldots,A_{\omega(m_{-}(\omega)+\din_\omega)}$ be the random $\din_\omega$ points given by a conditionally independent Poisson process on $[A_{\omega},1]$ with random intensity
	\eqn{
		\label{for:pointgraph:poisson}
		\rho_{\omega}(x) = {(1-\chi)}{\Gamma_\omega}\frac{x^{-\chi}}{A_\omega^{1-\chi}}.
	}
	Assign these ages to $\omega(m_{-}(\omega)+1),\ldots,\omega(m_{-}(\omega)+\din_\omega)$. {Assign them type $\Young$,} and set them unexplored;
	\item Draw an edge between $\omega$ and each one of the nodes $\omega 1,\ldots,\omega(m_{-}(\omega)+\din_\omega)$;
	\item Set $\omega$ as explored.
\end{enumerate}
We call the resulting tree the {\em P\'olya point tree with parameters $m$ and $\delta$,} and denote it by $\PPT(m,\delta)$. Occasionally we drop $m$ and $\delta$ when referring to $\PPT(m,\delta)$. The vertex-marks of $\PPT$ here are different from the ones in \cite{BergerBorgs}, but we can always retrieve these vertex-marks in the definition in \cite{BergerBorgs} from the vertex-marks used here and the other way round also.
\subsection{Percolation on large-set expanders}
We define several notations that will be used throughout the article. Let $G=(V(G),E(G))$ (as before) be an undirected multi-graph with self-loops and multi-edges. The degree of a vertex $u\in V$ is denoted by $d_G(u)$. Every self-loop contributes $2$ to the degree of the vertex. For any $S\subset V$, define the volume of the subset $S$ as
\eqn{\label{def:vol:subset}
	\vol_G(S)=\sum\limits_{v\in S}d_G(v)~,}
and the cutset of $S$ is defined as
\eqn{\label{def:cutset:subset}
	\cut_G(S,S^c) = \{ e\in E: e\quad \text{has one end in $S$ and the other end in $S^c$} \}~.}
For every $\varepsilon>0$, we define the edge-expander constant of the graph $G$ as
\eqn{\label{def:expander-constant:graph}
	\alpha(G,\varepsilon)=\min\limits_{\substack{S\subset V\\\varepsilon|V|\leq |S|\leq|V|/2}}\frac{|\cut_G(S,S^c)|}{|S|}~.}
We call a graph $G$ an $(\alpha,\varepsilon,d)$-large-set expander if the graph has average degree bounded above by $d$ and $\alpha(G,\varepsilon)\geq \alpha$. A sequence of (possibly random) graphs $(G_n)_{n\geq 1}$ is called a large-set expander sequence with bounded average degree if there exists a $d<\infty$, and for every $\varepsilon\in (0,1/2)$ there exists $\alpha>0$  such that,
\eqn{\label{def:large-set-expander-sequence}
	\prob(G_n \text{ is $(\alpha,\varepsilon,d)$-large-set expander})\to 1,\quad \text{as }n\to\infty~.
}

Alimohammadi, Borgs and Saberi \cite{ABS22} studied percolation on such sequences of large-set expanders with bounded average degree. Let $(G_n)_{n\geq 1}$ be a (possibly random) sequence of large-set expanders with bounded average degree and $\mu$ be a probability distribution on $\mathcal{G}_\star$. For any graph $G\in\mathcal{G}_\star$, we use $G(\pi)$ to denote the subgraph of $G$ obtained after performing bond percolation on $G$ with percolation probability $\pi$. Suppose $(G_n)_{n\geq 1}$ converges locally to $(G,o)\in\mathcal{G}_\star$ with law $\mu$. We define $\Ccal_i^{\sss(n)}(\pi)$ as the $i$-th largest component of the graph $G_n(\pi)$ (breaking ties arbitrarily) and $\mathscr{C}(\pi)$ as the cluster containing the vertex $o$ in $G(\pi)$. Let $\zeta(\pi):= \mu\big( |\mathscr{C}(\pi)|=\infty\big)$. According to \cite[Theorem~1.1]{ABS22}, if $\pi\mapsto\zeta(\pi)$ is continuous at all $\pi\in[0,1]$, then 
\eqn{\label{expander:percolation-result:ABS22}
	\frac{|\Ccal_1^{\sss(n)}(\pi)|}{n}\overset{\prob}{\to}\zeta(\pi),}
where $\overset{\prob}{\to}$ denotes convergence in probability with respect to both the random graph and percolation. Furthermore, for all $\pi\in[0,1]$, ${|\Ccal_2^{\sss(n)}(\pi)|}/{n}\overset{\prob}{\to}0$, and this convergence is uniform in $\pi$ if $\pi$ is bounded away from $0$ and $1$.

\subsection{Proof strategy} 
In this section, we briefly discuss the brief structure of the proof of Theorem~\ref{thm:main:theorem}. Fix $m\in\N\setminus{\{1\}}$ and $\delta>-m$. Let us consider an initial graph of size $2$, and one of the initial vertices has degree at most $m$. Without loss of generality consider $a_2\leq m$. Now we construct a preferential attachment model with parameters $m$ and $\delta$ and $G_{n}$ is as defined in Section~\ref{subsec:main-theorem}. 
Firstly, we establish the large-set expander property of models (a), (b) and (d) with bounded average degree, following the strategy of Mihail, Papadimitrou and Saberi \cite{MPS06}:
\begin{Proposition}\label{prop:Cheeger-value:model-b}
	Let $(G_n)_{n\geq 2}$ be either preferential attachment model (a), (b) and (d) sequence and $m\geq 2$. Then for any constant $\vep>0$, there exists $\alpha>0$ such that, $(G_{n})_{n\geq1}$ is $(\alpha,\vep,m)$-large-set expander. In particular, there exists $c=c(\infty)>0$ such that,
	\eqn{\label{eq:prop:Cheeger-value:model-b}
		\prob(\alpha(G_{n},\vep)<\alpha)= o(\e^{-cn})~.}
\end{Proposition}

Next, we delve into the analysis of the survival probability of the percolated P\'olya point tree, which is the local limit of our preferential attachment models. We investigate the survival probability of the percolated P\'olya point tree in two regime: $\delta\leq 0$ and $\delta>0$.
For $\delta\leq 0$, we prove the following proposition: 
\begin{Proposition}\label{prop:criticality:negative-delta}
	Fix $m\geq 2$ and $\delta\in(-m,0]$. Then, the critical percolation threshold of the P\'olya point tree with parameters $m$ and $\delta$ is $0$.
\end{Proposition}
The proof of this proposition uses an interesting property of the P\'olya point tree, namely the ``elbow children". We do not use any functional analytic tools to prove this proposition in contrast to \cite{ABS22}. We prove that for non-positive $\delta$ and any $\pi>0$, the P\'olya point tree, percolated with probability $\pi$, stochastically dominates a super-critical single-type branching process.

On the other hand, for $\delta>0$, the critical percolation threshold turns out to be positive, and we need to investigate the mean offspring generator of P\'olya point tree in detail. For any $m\geq 1$ and $\delta>-m$, from \cite[(5.4.86)]{vdH2} the kernel function of the mean offspring operator of P\'olya point tree is given by 
\eqn{\label{eq:kernel:offspring-operator:PPT}
	\kappa\big((x,s), (y,t)\big)=\frac{c_{st}(\one_{\{x>y, t=\old\}}+\one_{\{x<y, t=\young\}})}{(x\vee y)^{\chi}(x\wedge y)^{1-\chi}},
}
with $\chi=(m+\delta)/(2m+\delta)\in (0,1)$ and 
\eqn{
	\label{eq:def-cst}
	c_{st}=
	\begin{cases}
		\frac{m(m+\delta)}{2m+\delta} &\text{ for }st=\Old\Old,\\
		\frac{m(m+1+\delta)}{2m+\delta} &\text{ for }st=\Old\Young,\\
		\frac{(m-1)(m+\delta)}{2m+\delta} &\text{ for }st=\Young\Old,\\
		\frac{m(m+\delta)}{2m+\delta} &\text{ for }st=\Young\Young.
	\end{cases}
}
Instead of confining our analysis to the restricted space, we opt to extend the type space to $\Scal_e=[0,\infty)\times \{ \Old,\Young \}$ while keeping the kernel unchanged. This extension allows us to incorporate vertices that will join the preferential attachment graph after time $n$. By growing the $\PPT$ on this extended type-space and subsequently truncating all nodes with an age greater than $1$, we retrieve the $\PPT$. Let $\bar{\bfT}_{\kappa}$ and ${\bfT}_{\kappa}$ denote the mean offspring operator defined on $L^2(\Scal_e,\lambda_e)$ and $L^2(\Scal,\lambda)$, respectively, where $\lambda$ and $\lambda_e$ are the Lebesgue measures on the continuous part of $\Scal$ and $\Scal_e$, respectively.
The following theorem identifies an eigenvalue of $\bar{\bfT}_{\kappa}$ along with the corresponding eigenfunction, and this eigenvalue is the spectral norm of the mean offspring operator also irrespective of the type space being $\Scal$ or $\Scal_e$:
\begin{Theorem}[Spectral norm of the offspring generator]
	\label{thm:operator-norm:PPT}
	Fix $m\geq 1$ and $\delta>0$. Let $r(\bfT_{\kappa})$ and $r(\bar{\bfT}_{\kappa})$ denote the spectral norms of $\bfT_{\kappa}$ and $\bar{\bfT}_{\kappa}$, respectively. Then
	\eqn{
		\label{nu-equality-PAM}
		r({\bfT}_{\kappa})=r(\bar{\bfT}_{\kappa})=2\frac{m(m+\delta)+\sqrt{m(m-1)(m+\delta)(m+1+\delta)}}{\delta}.
	}
	Additionally, the explicit eigenfunction of $\bar{\bfT}_\kappa$ corresponding to $r(\bar{\bfT}_{\kappa})$ is $f(x,s)=\bfp_s/\sqrt{x}$, where $\bfp=(\bfp_{\old},\bfp_{\young})$ is the right eigenvector of $\bfM$ defined by
	\[
		\bfM=\begin{bmatrix}
			c_{\old\old} & c_{\old\young}\\
			c_{\young\old} & c_{\young\young}
		\end{bmatrix}~.
	\]
\end{Theorem}
It can further be shown that the operator norms of these operators are also equal and equal to their spectral norm.
Although we do not make use of this fact for proving the critical percolation threshold of P\'olya point tree, this result is interesting in its own right.


Next we prove that for $\pi\leq 1/r(\bar{\bfT}_{\kappa})$, the percolated P\'olya point tree is subcritical, showing the required left continuity of the survival probability function $\zeta$, whereas the right continuity follows in general for percolation models from \cite[Lemma~8.9]{Grimmet99}: 
\begin{Proposition}\label{lem:subcriticality}
	For $\delta>0$, the $\PPT$ dies out almost surely if it is percolated with probability $\pi\leq 1/r(\bar{\bfT}_\kappa)$.
\end{Proposition}
Alternatively, for $\pi>1/r(\bar{\bfT}_{\kappa})$, we prove that percolated P\'olya point tree is supercritical.
\begin{Proposition}[Supercriticality of $\PPT$]\label{prop:supercriticality}
	For $\delta>0~$, the $\PPT$ survives with positive probability when percolated with probability $\pi>1/r(\bar{\bfT}_{\kappa})$.
\end{Proposition}
We prove this proposition adapting the proof of \cite[Lemma~3.3]{DM13},
utilizing a spine decomposition argument akin to that used by Lyons, Pemantle and Peres in \cite{LPP95}.
Now we have all the tools to prove Theorem~\ref{thm:main:theorem}:
\begin{proof}[Proof of Theorem~\ref{thm:main:theorem}]
	The proof of Theorem~\ref{thm:main:theorem} is quite straightforward from \cite[Theorem~1.1]{ABS22} and the tools we have developed so far. According to \cite[Theorem~1.1]{ABS22}, the critical percolation threshold of the sequence of random graphs is equal to that of its local limit and the giant component is unique, if the following properties are satisfied:
	\begin{enumerate}
		\item \label{cond:1:expander} For any $\vep>0$, there exists $\alpha>0$ and $m\in\N$ such that the sequence of (possibly random) graphs is an $(\alpha,\vep,m)$-large-set expander;
		\item \label{cond:2:continuity} The survival probability of the local limit of the graph sequence is continuous at its critical percolation threshold.
	\end{enumerate}
	Proposition~\ref{prop:Cheeger-value:model-b} show that models (a), (b) and (d) are large-set expanders, thus proving that Condition~\ref{cond:1:expander} holds. Proposition~\ref{prop:criticality:negative-delta} shows that $\pi_c=0$ for the $\PPT$ with $\delta\in(-m,0]$. For $\delta>0$, Proposition~\ref{lem:subcriticality} and \ref{prop:supercriticality} show that the inverse of the spectral norm of the mean offspring operator is the critical percolation threshold of the P\'olya point tree, and its survival probability is continuous at the critical percolation threshold. Thus condition~\ref{cond:2:continuity} holds. Therefore, the critical percolation thresholds for preferential attachment models (a), (b) and (d) are equal to that of P\'olya point tree. Lastly, Theorem~\ref{thm:operator-norm:PPT} identifies the explicit expression of the spectral norm of the mean offspring operator of P\'olya point tree, thus completing the proof.
\end{proof}
\subsection{Discussions and open problems}
We believe that our proof can be extended to other preferential attachment models. For example, a similar technique may prove useful in identifying critical percolation thresholds for preferential attachment models with i.i.d.\, random out-edges as studied in \cite{DEH09,GV17,RRR22}. Our proof of the large-set expander property for preferential attachment model extensively uses the fact that $m\geq 2$. However we stress that all the lemmas and propositions for the P\'olya point tree with $\delta>0$ hold true for $m=1$ as well. 

Although we identify the exact critical percolation thresholds for several preferential attachment models, our techniques do not reveal the order of the phase transition. Alimohammadi, Borgs and Saberi have proved that this phase transition is {infinite order} for $\delta=0$, i.e., $\zeta(\pi)$ is infinitely differentiable at $\pi_c$. Eckhoff, M\"orters and Ortgiese proved in \cite{EMO21} that the phase transition in Bernoulli preferential attachment models defined in \cite{DM13} is of infinite order also. Therefore, we believe that the phase transition is of infinite order for other preferential attachment models and any admissible $\delta$.
\subsection{Organisation of the article}
In Section~\ref{sec:large_set_expander}, we establish that the preferential attachment models are large-set expanders with bounded average degree by proving Proposition~\ref{prop:Cheeger-value:model-b}.
Since we only require an eigenvalue of $\bar{\bfT}_\kappa$ and the eigenfunction corresponding to that particular eigenvalue, we defer the proof of Theorem~\ref{thm:operator-norm:PPT} to Section~\ref{sec:spectral_radius} but use the fact that $f(x,s)$ in Theorem~\ref{thm:operator-norm:PPT} is an eigenfunction of $\bar{\bfT}_\kappa$ corresponding to the eigenvalue $r(\bar{\bfT}_\kappa)$.
In Section~\ref{sec:percolation:threshold}, we identify the inverse of the spectral norm as the critical percolation threshold of the P\'olya point tree. Furthermore, we prove that the survival function of the P\'olya point tree is continuous at the critical percolation threshold. Lastly,
in Section~\ref{sec:spectral_radius}, we explicitly compute the operator and spectral norm of the mean offspring operator of the P\'olya point tree for positive $\delta$ and prove Theorem~\ref{thm:operator-norm:PPT}.
\section{Large-set expander property}\label{sec:large_set_expander}
In this section, we prove Proposition~\ref{prop:Cheeger-value:model-b}.
For that, we show that for all $\varepsilon\in(0,1/2)$, there exists $\alpha>0$ such that sequence of preferential attachment graphs are $(\alpha,\vep,m)$-large-set expanders. For this, we adapt the Cheeger value calculation done for model (b) and $\delta=0$ from \cite{MPS06}, to {models (a), (b) and (d)} and all admissible $\delta$. 
Let us denote $\{1,2,\ldots,n\}$ by $[n]$ for any $n\in\N$.

According to \cite[Chapter~5]{vdH2}, the graph $G_{n}$ can be constructed by \emph{collapsing} of a preferential attachment tree $T_{m(n-2)+a_{[2]}}$, consisting of $m(n-2)+a_{[2]}$ vertices. 
Let us  denote 
\eqn{\label{eq:def:total-degree}
m_{[n]}=m(n-2)+a_{[2]}~.}
These vertices of $T_{m_{[n]}}$ are referred to as \emph{mini-vertices}. A mini-vertex $u$ is associated with a vertex $v\in[n]$ if $u$ belongs to the set of mini-vertices that form the vertex $v$ after collapsing.


In summary, the construction of $G_{n}$ involves collapsing a preferential attachment tree, leading to mini-vertices being collapsed to vertices. The initial graph is then obtained by arranging the mini-vertices with self-loops before those creating connections between specific vertices.

In the tree $T_{m_{[n]}}$, mini-vertex $u$ connects to mini-vertex $n(u)$. We shall use this pre-collapsed version of the graph $G_{n}$ for proving Proposition~\ref{prop:Cheeger-value:model-b}. Similarly as in \cite{MPS06}, we call a set $S\subset[n]$ is \rm{BAD} if 
\eqn{\label{def:bad-set}
	|\cut_{\sss G_{n}}(S,S^c)|<\alpha |S|~.}
For convenience, we also recall other notations used in \cite{MPS06}:
\begin{Definition}[\rm{FIT} and \rm{ILL} mini-vertices]
	\label{def:fit-ill-minivertex}
	Fix $S\subset[n]$. Then we call a mini-vertex $t\in [m_{[n]}]$ \rm{FIT} when
	\begin{itemize}
		\item[${\rhd}$] $t$ is associated to a vertex in $S$ and $n(t)$ is associated to a vertex in $S^c$;
		\item[${\rhd}$] $t$ is associated to a vertex in $S^c$ and $n(t)$ is associated to a vertex in $S$.
	\end{itemize}
	A mini-vertex is \rm{ILL} if it is not \rm{FIT}. 
\end{Definition}
Note that the \rm{FIT} mini-vertices create the connections between $S$ and $S^c$, whereas \rm{ILL} mini-vertices are responsible for creating edges between two vertices of the same set, either $S$ or $S^c$. Therefore, bounding the number of \rm{FIT} mini-vertices will provide us a bound on $|\cut_{\sss G_{n}}(S,S^c)|$. We prove the following intermediate lemma that upper bounds the probability of observing too few \rm{FIT} mini-vertices and which will be helpful for proving the Proposition~\ref{prop:Cheeger-value:model-b}:
\begin{Lemma}[Upper bound on probability for mini-vertices to be \rm{ILL} in models (a), (b) and (d)]
	\label{lem:intermediate:Cheeger-value}
	For preferential attachment models (a), (b) and (d) and a fixed subset $S\subset[n]$ of cardinality $k$ and for a fixed subset $A\subset [m_{[n]}]$ such that $|A|\leq \alpha k$, the probability that all mini-vertices in $[m_{[n]}]\setminus A$ are {\rm{ILL}} and all mini-vertices in $A$ are {\rm{FIT}} with respect to $S$, is at most $(mn+\delta_0)^{\delta_0}\binom{mk}{\alpha k}/\binom{m_{[n]}-\alpha k}{mk-\alpha k}$ for some $\delta_0$ fixed.
\end{Lemma}
 Since the proof of Lemma~\ref{lem:intermediate:Cheeger-value} is a minor extension of the proof of \cite[Lemma~2]{MPS06}, applying to $\delta=0$ and model (b), to models (a) (b) and (d) with general $\delta>-m$, we defer its proof to the appendix and prove Proposition~\ref{prop:Cheeger-value:model-b} subject to Lemma~\ref{lem:intermediate:Cheeger-value}. This proof also follows similar line of proof as in \cite[Proof of Theorem 1]{MPS06}. This step requires us to choose $m\geq 2$.
\begin{proof}[Proof of Proposition~\ref{prop:Cheeger-value:model-b}]
	The proof of Proposition~\ref{prop:Cheeger-value:model-b}	 follows the proof of \cite[Theorem~1]{MPS06} almost verbatim. Let $S\subset [n]$ be a set of vertices of size $k$ and $A\subset [m_{[n]}]$ be the set \rm{FIT} mini-vertices such that $|A|\leq \alpha k$ for some $\alpha\in(0,1)$. By Lemma~\ref{lem:intermediate:Cheeger-value},
	\eqan{\label{for:prop:Cheegar-value:model-b:1}
	\prob\Big( \bigcap\limits_{t\in[m_{[n]}]\setminus A}\{t ~\text{is \rm{ILL}}\} \Big)\leq& (mn+\delta_0)^{\delta_0} \binom{mk}{|A|}\binom{m_{[n]}-|A|}{mk-|A|}^{-1}\nn\\
	\leq& (mn+\delta_0)^{\delta_0} \binom{mk}{\alpha k}\binom{m_{[n]}-\alpha k}{mk-\alpha k}^{-1}~,
	}
	where the last inequality follows from the fact that $|A|\leq \alpha k$. Now we have $\binom{n}{k}$ many choices for $S$, whereas for fixed $S$ and $|A|$ being at most $\alpha k$, we have $\alpha k\binom{m_{[n]}}{\alpha k}$ many choices for $A$. Therefore, \eqref{for:prop:Cheegar-value:model-b:1} implies that
	\eqan{\label{for:prop:Cheegar-value:model-b:2}
	\prob\Big( \alpha(G_{n},\vep)<\alpha \Big)\leq& (mn+\delta_0)^{\delta_0} \sum\limits_{k=\vep n}^{n/2} \binom{n}{k}\alpha k\binom{m_{[n]}}{\alpha k}\frac{\binom{mk}{\alpha k}}{\binom{m_{[n]}-\alpha k}{mk-\alpha k}}~.
	}
	In the next few step, we bound the summands in the RHS of \eqref{for:prop:Cheegar-value:model-b:2}.
	Using that $\binom{n}{k}\binom{m_{[n]}-n-\alpha k}{mk-k-\alpha k}\leq \binom{m_{[n]}-\alpha k}{mk-\alpha k}$, we upper bound the summand in the RHS of \eqref{for:prop:Cheegar-value:model-b:2} by
	\eqn{\label{for:prop:Cheegar-value:model-b:03}
	\alpha k\binom{m_{[n]}}{\alpha k}\frac{\binom{mk}{\alpha k}}{\binom{m_{[n]}-n-\alpha k}{mk-k-\alpha k}}~.}
	Since binomial coefficients can be upper and lower bounded by
	\eqn{\label{for:prop:Cheegar-value:model-b:04}
	\Big(\frac{r}{s}\Big)^{s}\leq \binom{r}{s}\leq \Big(\frac{\e r}{s}\Big)^{s}~,}
	we can upper bound the expression in \eqref{for:prop:Cheegar-value:model-b:03} by
	\eqan{\label{for:prop:Cheegar-value:model-b:05}
	&~\alpha k\Big( \frac{\e^2 m_{[n]}m}{\alpha^2 k} \Big)^{\alpha k}\Big( \frac{(m-1)k-\alpha k}{m_{[n]}-n-\alpha k} \Big)^{(m-1-\alpha)k}\nn\\
	\leq&~ \alpha k\Big( \frac{\e^2 m_{[n]}m}{\alpha^2 k} \Big)^{\alpha k}\Big( \frac{(m-1)k}{m_{[n]}-n} \Big)^{(m-1-\alpha)k}~.
	}
	Using the explicit expression for $m_{[n]}$ from \eqref{eq:def:total-degree}, we simplify the RHS of \eqref{for:prop:Cheegar-value:model-b:05} as
	\eqn{\label{for:prop:Cheegar-value:model-b:06}
	\alpha k \Big( \frac{\e m}{\alpha}\Big)^{2\alpha k}\Big( \frac{n+(a_{[2]}-2m)/m}{k} \Big)^{\alpha k}\Big( \frac{k}{n+(a_{[2]}-2m)/(m-1)} \Big)^{(m-1-\alpha)k}~.
	}
	Denoting $\delta_1=\min\{(a_{[2]}-2m)/m,(a_{[2]}-2m)/(m-1)\}$, we upper bound the expression in \eqref{for:prop:Cheegar-value:model-b:06} by
	\eqn{\label{for:prop:Cheegar-value:model-b:07}
	\alpha k \Big( \frac{\e m}{\alpha}\Big)^{2\alpha k}\Big( \frac{k}{n+\delta_1} \Big)^{(m-1-2\alpha)k}~.
	}
	Hence,
	\eqan{\label{for:prop:Cheegar-value:model-b:08}
		\prob\Big( \alpha(G_{n},\vep)<\alpha \Big)
		\leq &(mn+\delta_0)^{\delta_0} \sum\limits_{k=\vep n}^{n/2}\alpha k \Big( \frac{\e m}{\alpha}\Big)^{2\alpha k}\Big( \frac{k}{n+\delta_1} \Big)^{(m-1-2\alpha)k}~.
	}
	The inequalities in \eqref{for:prop:Cheegar-value:model-b:03}-\eqref{for:prop:Cheegar-value:model-b:07} follow exactly from the calculations in \cite[Proof of Theorem~1]{MPS06}. There are $O(n)$ many summands in the sum. Thus, to prove the proposition, it suffices to bound the leading term by $o(\e^{-nc})$ for some $c>0$. Thus, now we study the function
	\eqn{\label{for:prop:Cheegar-value:model-b:3}
	g(x)=\alpha x \Big( \frac{\e m}{\alpha} \Big)^{2\alpha x} \Big( \frac{x}{n+\delta_1} \Big)^{(m-1-2\alpha)x}~,}
	for $x\in[\vep n,n/2]$. We argue that there exists $x_0\in[1,n/2]$ such that $g(x)$ is monotonically decreasing for $x\in[1,x_0]$ and monotonically increasing for $x\in[x_0,n/2]$. For this we first calculate the derivative of $g(x)$ as
	\eqn{\label{for:prop:Cheegar-value:model-b:4}
	g^\prime(x)=\frac{g(x)}{x}\Big[ 1+x\Big( 2\alpha \Big(1+\log\frac{m}{\alpha}\Big)+(m-1-2\alpha)\Big(\log\frac{x}{n+\delta_1}+1\Big) \Big) \Big]~.
	}
	Therefore, the sign of $g^\prime(x)$ is determined by the sign of 
	\eqn{\label{for:prop:Cheegar-value:model-b:5}
	g_1(x)=\frac{xg^\prime(x)}{g(x)}=1+x\Big( 2\alpha \Big(1+\log\frac{m}{\alpha}\Big)+(m-1-2\alpha)\Big(\log\frac{x}{n+\delta_1}+1\Big) \Big)~.}
	Note that $g_1(1)<0$ and $g_1(n/2)>0$ for $n$ sufficiently large, while the first and second derivative of $g_1$ are
	\eqan{\label{for:prop:Cheegar-value:model-b:6}
	g_1^\prime(x)=&2\alpha \Big(1+\log\frac{m}{\alpha}\Big)+(m-1-2\alpha)\Big(\log\frac{x}{n+\delta_1}+1\Big)~,\nn\\
	g_1^{\prime\prime}(x)=&\frac{m-1-2\alpha}{x}>0~.}
	Therefore, $g_1$ is a convex function having a unique root in $[1,n/2]$. Consequently for some $x_0\in[1,n/2],~g(x)$ is decreasing in the interval $[1,x_0]$ and increasing in the interval $[x_0,n/2]$. We conclude that it is enough to show that $\max\{g(\vep n),g(n/2)\}=o(\e^{-nc}).$ Note that
	\eqan{\label{for:prop:Cheegar-value:model-b:7}
	g(\vep n)=&\alpha (\vep n)\Big[\Big( \frac{\e m}{\alpha} \Big)^{2\alpha} \vep^{(m-1-2\alpha)}\Big]^{\vep n}~,\nn\\
	\text{and}\qquad g(n/2)=&\alpha (n/2) \Big[\Big( \frac{\e m}{\alpha} \Big)^{2\alpha} \Big( \frac{1}{2} \Big)^{(m-1-2\alpha)}\Big]^{n/2}~.
	}
	For any $r\in(0,1),~g(r(n+\delta_1))$ drops exponentially first when 
	\eqn{\label{for:prop:Cheegar-value:model-b:8} 
	\Big( \frac{\e m}{\alpha} \Big)^{2\alpha}<r^{2\alpha-(m-1)}~.}
	Taking logarithms on both sides of \eqref{for:prop:Cheegar-value:model-b:8}, the condition simplifies to
	\eqn{\label{for:prop:Cheegar-value:model-b:9}
	\alpha< \frac{(m-1)\log \frac{1}{r}+2\alpha\log \alpha}{\log m-\log r +1}~.
	} 
	Since $2\alpha\log\alpha\to 0$ as $\alpha\to 0$ and for any $\gamma\leq 1/2$, we can choose $\alpha_\gamma>0$ such that 
	\[
		\alpha_\gamma\log \alpha_\gamma =\frac{m-1}{4}\log \gamma<0~. 
	\]
	By \eqref{for:prop:Cheegar-value:model-b:7}-\eqref{for:prop:Cheegar-value:model-b:9}, both $g(\vep n)$ and $g(n/2)$ are $o(\e^{-cn})$, for some $c>0$, if we choose 
	\[
		\alpha<\min\Big\{ \alpha_\vep,\alpha_{1/2},\frac{(m-1)\log \frac{1}{\vep}+2\alpha_\vep\log \alpha_\vep}{\log m-\log \vep +1},\frac{(m-1)\log 2+2\alpha_{1/2}\log \alpha_{1/2}}{\log m-\log \frac{1}{2} +1} \Big\}~.
	\]
	Therefore, the RHS of \eqref{for:prop:Cheegar-value:model-b:2} converges to $0$ as $n\to\infty$ for a suitable choice of positive $\alpha$ (depends only on $\vep$ and $m$). Hence $(G_{n})_{n\geq 2}$ is whp an $(\alpha,\vep,m)$-large-set expander. 
\end{proof}
The preferential attachment models (a), (b) and (d) have average degree $m$. Proposition~\ref{prop:Cheeger-value:model-b} proves that for any $\vep>0$, there exists an $\alpha>0$, such that the sequence of preferential attachment models ((a), (b) or (d) and any $\delta>-m$) is $(\alpha,\varepsilon,m)$-large-set expander.
\section{The local limit and percolation threshold}\label{sec:percolation:threshold}
 In \cite[Appendix~E]{ABS22}, the authors have showed that the survival probability is continuous at $\pi_c$ for the $\delta=0$ case of the P\'olya point tree. Additionally, the authors have proved in  \cite[Corollary~2.2]{ABS22} that $\zeta(\pi)$ is continuous for all $\pi\neq \pi_c$.
We now proceed to identify the critical percolation threshold $\pi_c$ for the P\'olya point tree, a multi-type branching process with a mixed continuous and discrete type space. Additionally, we demonstrate that $\zeta(\pi)$ is continuous at $\pi_c$ for any $\delta>-m$. Intuitively, $\pi_c$ for such branching processes is the inverse of the spectral norm of the offspring operator but this requires a proof.

Let $r(\bar{\bfT}_{\kappa})$ denote the spectral norm of the offspring operator $\bar{\bfT}_{\kappa}$. The following theorem presents the critical percolation threshold of the P\'olya point tree:
\begin{Theorem}[Critical percolation threshold for the P\'olya point tree]\label{thm:critical-percolation:PPT}
	Fix $m\geq 1$ and $\delta>-m$. Let $\bar{\bfT}_\kappa$ be the offspring operator of the P\'olya point tree on $\Scal$. Then, for $\delta>0$, the critical percolation threshold of the P\'olya point tree is $1/r(\bar{\bfT}_\kappa)$. For $m\geq2$ and $\delta\in(-m,0]$, the critical percolation threshold is $0$. Additionally, $\pi\mapsto\zeta(\pi)$, the survival probability of the P\'olya point tree after percolation, is continuous at the critical percolation threshold.
\end{Theorem}
Theorem~\ref{thm:critical-percolation:PPT} identifies the critical percolation threshold for the P\'olya point tree and establishes the continuity of the survival probability at $\pi_c$ for $\delta>-m$.
Consequently, we obtain the critical percolation threshold for the preferential attachment graphs, which is the same as that of the P\'olya point tree, its local limit.

We start by proving the theorem for a non-positive $\delta$. This part turns out to be easier to prove than the other regime of $\delta$. Here we provide a proof for $\delta=0$ also, distinct from the functional analytic approach used in \cite{ABS22}.

\subsection{Non-positive $\delta$ regime}\label{subsec:non-negative-delta}
In \cite{ABS22}, authors have treated the $\delta=0$ case using fixed-point approximation techniques. We present an alternative approach where we establish that for any $\delta\leq0$ and $\pi>0$, the P\'olya point tree stochastically dominates a single-type discrete branching process with a mean offspring larger than $1$. Before defining this branching process, we define the ``elbow children" as follows:
\begin{Definition}[Elbow children]
	Let $(x,s)$ be a node in $\PPT.~(z,{\Old})$ is called an \emph{elbow child} of $(x,s)$ if there exists $(y,{\Young})$ in $\PPT$ such that $(x,s)$ connects to $(y,{\Young})$ and $(y,{\Young})$ connects to $(z,{\Old})$ in $\PPT$. The edge-connections involved in forming an elbow child are called \emph{elbow edges}.
\end{Definition}
\begin{wrapfigure}{r}{3cm}
\begin{tikzpicture}[line cap=round,line join=round,>=triangle 45,x=1.0cm,y=1.0cm]
	\clip(-7.5,2.2) rectangle (-5,4.7);
	\draw [line width=2.pt,color=wrwrwr] (-7.,3.)-- (-6.5,4.);
	\draw [line width=2.pt,color=wrwrwr] (-6.5,4.)-- (-6.,2.5);
	\draw (-6,2.8) node[anchor=north west] {$(z,{\Old})$};
	\draw (-7,4.7) node[anchor=north west] {$(y,{\Young})$};
	\draw (-7.5,2.9) node[anchor=north west] {$(x,s)$};
	\begin{scriptsize}
		\draw [fill=rvwvcq] (-7.,3.) circle (2.5pt);
		\draw [fill=rvwvcq] (-6.5,4.) circle (2.5pt);
		\draw [fill=rvwvcq] (-6.,2.5) circle (2.5pt);
	\end{scriptsize}
\end{tikzpicture}

	\label{fig:elbow-child}
\end{wrapfigure}
In the figure, the $(z,{\Old})$ node is connected to $(x,s)$ through $(y,{\Young})$. Thus, $(z,\Old)$ is an elbow child of $(x,s)$. The edge-connections $((x,s),(y,\Young))$ and $((y,\Young),(z,\Old))$ are elbow edges. Note that all the elbow children of any node have label $\Old$ in the tree. Further, the set of elbow children of nodes at $n$-th generation lower bounds the total number of $\Old$ labeled nodes in the $(n+2)$-th generation in a $\PPT$.

Fix any $\pi>0$. We now work with the $\PPT$ percolated with probability $\pi$. To construct the branching process, we proceed in two steps. Firstly, we construct a branching process with elbow children, where the offspring generation still depends on the age of the parent node. Next, we further stochastically lower bound it with a branching process, where the offspring generation process does not depend on the age of the parent node.

Let $h>0$ be a fixed threshold. Starting from the root in the percolated $\PPT$, there is a positive probability of obtaining a child $(x,\Old)$ of the root such that $x<h$. We consider the branching process $\Tcal(x,\Old)$, starting at $(x,\Old)$. The children of $(x,\Old)$ in $\Tcal(x,\Old)$ are given by the elbow-children of $(x,\Old)$ in the percolated $\PPT$ we started with. We continue growing the tree following the same procedure. It is evident that the size of $\Tcal(x,\Old)$ serves as a lower bound for the size of the percolated $\PPT$. We then prune the tree such that all nodes have an age at most $h$. In this pruned tree, all nodes have label $\Old$, making the type labels redundant. We denote this pruned tree as $\Tcal^{\sss (h)}(x)$ and the nodes no longer have labels. From its construction, it is evident that the size of $\Tcal^{\sss (h)}(x)$ further lower bounds that of $\Tcal(x,\Old)$ and in turn, lower bounds the size of the percolated $\PPT$.

Before proceeding to the next step of our stochastic lower bounds, we compute the explicit kernel function $\kappa_1$ of the offspring operator of $\Tcal^{\sss (h)}(x)$. Since $(z,\Old)$ is an elbow child of $(y,\Old)$, the latter connects to the former through a $\Young$ labeled node.
Therefore, $\kappa_1$ is given by
\eqn{\label{def:kernel:lower-bound:1}
\kappa_1(y,z)=\pi^2\int\limits_{y}^1\kappa((y,\Old),(u,\Young))\kappa((u,\Young),(z,\Old))\,du~.}
Here we crucially use the fact that the node $(u,\Young)$ connecting $(y,\Old)$ and $(z,\Old)$ in the $\PPT$ is unique if it exists. The $\pi^2$ factor is due to bond percolation on the two elbow-edges used to connect $(y,\Old)$ and $(z,\Old)$ in the $\PPT$. The RHS of \eqref{def:kernel:lower-bound:1} can be simplified as
\eqn{\label{eq:kernel:lower-bound:1}
\kappa_1(y,z)=\begin{cases}
	c_{\old\young}c_{\young\old}\pi^2 (yz)^{-1+\chi}(1-y^{1-2\chi})/(1-2\chi)\hspace{0.9cm}\text{for }\delta<0,\\
	c_{\old\young}c_{\young\old}\pi^2 (yz)^{-1+\chi}\log(1/x)\hspace{2.78cm}\text{for }\delta=0,
\end{cases}}
where $\chi=(m+\delta)/(2m+\delta)$ as defined earlier. Now we move on to the second step of the proof. In this step, we aim to establish a lower bound for $\Tcal^{\sss (h)}(x)$ by comparing it to a single-type branching process where the offspring generation does not depend on the age of the parent node. To achieve this, we carefully choose the offspring distribution in such a way that we obtain the desired lower bound. To obtain the offspring distribution, we use a stochastic domination argument. Let $\Bcal_1^{\sss(h)}(x)$ denote the $1$-neighborhood of $x$ in $\Tcal^{\sss (h)}(x)$, and for any two random variables $X$ and $Y$, we use $X\preceq Y$ to denote that $Y$ stochastically dominates $X$.
\begin{Lemma}\label{lem:stoch-dom}
	Let $0<z_1\leq z_2<1$. Then $|\Bcal_1^{\sss(h)}({z}_2)|\preceq |\Bcal_1^{\sss(h)}(z_1)|$.
\end{Lemma}
\begin{proof}
	To prove the stochastic domination, we revisit the definition of the $\PPT$. From \eqref{for:pointgraph:poisson}, the ages $\Young$ labeled children of a node $(z,\Old)$ are generated following an inhomogeneous Poisson process with intensity
	\begin{equation}\label{for:lem:stoch-dom:1}
		\rho_{(z,\old)}(x) = (1-\chi)\Gamma_{(z,\old)} \frac{x^{-\chi}}{z^{1-\chi}} \one_{\{x \geq z\}},
	\end{equation}
	where $\Gamma_{(z,\old)}$ is a Gamma random variable that is independent of the choice of $z$. 
	Let $z_{ij}$ denote the age of the $j$-th $\Young$ labeled child of $(z_i,\Old)$.
	
	Since $\rho_{(z_2,\Old)}(x)\leq \rho_{(z_1,\Old)}(x)$ for all $x\geq z_1,~z_{11}$ is stochastically dominated by $z_{21}$. As a consequence of this stochastic domination between the ages, it can similarly be shown that $z_{1j}\preceq z_{2j}$ for all $j$. Here, it is important to remember that the $\Old$ labeled children of $(z_{ij},\Young)$ are the elbow children of $(z_i,\Old)$. From the construction of the $\PPT$,  for all $k\in[m-1]$, the age of the $k$-th $\Old$ labeled child of $(z_{ij},\Young)$ is distributed as $U^{1/\chi}z_{ij}$, where $U$ is an uniform random variable on $[0,1]$. Hence, for all $k\in[m-1]$, the age of the $k$-th $\Old$ labeled child of $(z_{1j},\Young)$ is stochastically dominated by that of $(z_{2j},\Young)$. Obviously, all these stochastic domination results also hold true for the percolated $\PPT$.
	
	The lemma follows from the fact that in $\Tcal^{\sss(h)}(z_1)$ and $\Tcal^{\sss(h)}(z_2)$, all the nodes with age higher than $h$ are discarded. 
\end{proof}
Lemma~\ref{lem:stoch-dom} implies that the offspring distribution of $z_1$ dominates that of $z_2$ in $\Tcal^{\sss(h)}(x)$ for $z_1\leq z_2$. Since the age of all nodes in $\Tcal^{\sss(h)}(x)$ is bounded from above by $h$, the offspring distribution of the node with age $h$ (denoted as $F_{\pi,h}$) is stochastically dominated by the offspring distribution of each of the nodes in $\Tcal^{\sss(h)}(x)$.

\begin{proof}[Proof of Theorem~\ref{thm:critical-percolation:PPT} for non-positive $\delta$]
	Consider a single-type discrete branching process ${\sf{BP}}(\pi, h)$ with offspring distribution $F_{\pi,h}$. From its construction, the size of ${\sf{BP}}(\pi, h)$, denoted as $|{\sf{BP}}(\pi, h)|$, serves as a stochastic lower bound on the size of the percolated $\PPT$.
	
	The expected number of offspring of every node in ${\sf{BP}}(\pi, h)$ can be expressed as
	\begin{equation}\label{eq:offspring:GWBP:1}
		\int\limits_0^h \kappa_1(h,u) \,du =
		\begin{cases}
			c_{\old\young}c_{\young\old}\pi^2 h^{-1+2\chi} {(1-h^{1-2\chi})}/{(\chi(1-2\chi))} & \text{for } \delta < 0,\\
			c_{\old\young}c_{\young\old}\pi^2 \log(1/h) & \text{for } \delta = 0.
		\end{cases}
	\end{equation}
	
	Since $2\chi < 1$ for $\delta < 0$, and $\log(1/h)\to\infty$ as $h\searrow0$, for every fixed $\pi>0$, we can choose the threshold $h$ so small that the RHS of \eqref{eq:offspring:GWBP:1} is greater than $1$, making ${{\sf{BP}}}(\pi, h)$ a supercritical branching process. Consequently, the percolated P\'olya point tree ($\PPT$) also survives with positive probability after percolating with probability $\pi$.
\end{proof}

\subsection{Positive $\delta$ regime}\label{subsec:positive-delta}
We present the proof of Theorem~\ref{thm:critical-percolation:PPT} for positive $\delta$ in two crucial steps. Firstly, we establish that the spectral norm of $\bar{\bfT}_{\kappa}$ corresponds to the inverse critical percolation threshold for the P\'olya point tree. Next, we prove the continuity of $\pi\mapsto\zeta(\pi)$ at the previously obtained critical percolation threshold. Our proof technique for the first step enables us to demonstrate the left continuity of $\zeta(\pi)$ at $\pi_c$. Furthermore, we use \cite[Lemma~8.9]{Grimmet99} to establish the right continuity of $\zeta(\pi)$ at $\pi_c$, thereby proving its continuity at $\pi_c$. 

We proceed to prove that the inverse of the spectral norm of the offspring generator is identical to the $\PPT$'s critical percolation threshold. This proof consists of two parts. First, we show that when $\pi \leq 1/r(\bar{\bfT}_k)$, the process dies out almost surely. To demonstrate this, we adapt an argument used in \cite[Lemma~3.3]{DM13}. 

Throughout this section, we work with the $\PPT$ with type space $\Scal_e$. In Proposition~\ref{lem:subcriticality}, we prove that when percolated with probability $\pi\leq 1/r(\bar{\bfT}_{\kappa})$, the probability that the percolated $\PPT$ has the age of the left-most node below $1$ at some generation converges to $0$. Therefore, after truncating the tree at nodes with age more than $1$, the tree will die out eventually. On the other hand, when the $\PPT$ is percolated with probability $\pi>1/r(\bar{\bfT}_{\kappa})$, we prove in Proposition~\ref{prop:supercriticality} that with some positive probability there exists an ancestral line of nodes where the ages of the nodes converge to $0$ and hence the $\PPT$ survives even in the restricted type space $\Scal$. 
\begin{proof}[Proof of Proposition~\ref{lem:subcriticality}]
	In Theorem~\ref{thm:operator-norm:PPT}, we have derived the eigenfunction of the offspring operator, corresponding to its spectral norm, denoted as $f(x, s)$. Let $Y_{(x, s)}^{(n)}(y, t)\,dy$ represent the empirical measures of the type and age of all offspring in the $n$-th generation of a percolated P\'olya point tree initiated at $(x, s)$.
	With each generation $n\geq 0$, we associate the \emph{score}
	\eqn{\label{eq:def:supermartingale}
		X_\pi^{(n)}(x,s):=\sum\limits_{t\in\{ {\old,\young} \}} \int\limits_0^\infty Y_{(x,s)}^{(n)}(y,t)\frac{f(y,t)}{f(x,s)}\,dy~. }
	Given that $\pi\leq 1/r(\bar{\bfT}_\kappa)$ and $f$ is the eigenfunction of $\bar{\bfT}_{\kappa}$ corresponding to $r(\bar{\bfT}_{\kappa})$, it follows that $(X_\pi^{(n)})_{n\geq 1}$ constitutes a non-negative super-martingale and therefore converges almost surely. Now, let us consider a fixed value of $N\in \R^{+}$.
	Define the event $\mathcal{A}_n(N)$ as
	\eqn{\label{for:lem:subcriticality:2}
	\Acal_n(N)=\{\exists\text{ particle in the $n$-th generation with age less than }N\}~.}
	To complete the proof of Lemma~\ref{lem:subcriticality}, we prove that $\Acal_n(N)$ happens finitely often almost surely, i.e.,
	\eqn{\label{for:lem:subcriticality:3}
	\prob(\Acal_n(N) ~\text{i.o.})=0~,}
	where i.o. means infinitely often. For the simplicity of notation, we shall use $\Acal_n$ to denote $\Acal_n(N)$ throughout this proof. To prove \eqref{for:lem:subcriticality:3}, we define two sequences of stopping times $T_k$ and $S_k$ by
	\eqan{\label{for:lem:subcriticality:02}
	T_k=&\inf\{n>k:|X_\pi^{(n)}-X_\pi^{(n-1)}|>1\}~,\\
	S_k=&\inf\{n>k:\Acal_n \text{ occurs}\}~,\nn
	}
	for $k\geq 1$. We now prove that $\prob\big(|X_\pi^{(n)} - X_\pi^{(n-1)}| > 1 \mid \mathcal{A}_{n-1}\big)$ is positive and independent of $n$, its value depending only on $N$. 
	We couple two samples of the offspring of $(n-1)$-st generation, where we keep all offspring identical except for the left-most child of the left-most node in the $(n-1)$-st generation. The two left-most children of the left-most node in the $(n-1)$-st generation are sampled independently. Let $(A_1,s_1)$ and $(\bar{A}_1,\bar{s}_1)$ be the two copies of the left-most child of the left-most node in the $(n-1)$-st generation.
	
	Note that the distribution of $(A_1,s_1)$ does not depend on $n$, but only on the fact that its parent node has age less than $N$. Since the parent node of $(A_1,s_1)$ has age at most $N$ conditionally on $\Acal_{n-1}$,
	
	\eqn{\label{for:lem:subcriticality:03}
	\prob\Big(\inf\limits_{s,s_1,\bar{s}_1\in\{\old,\young\}}\big|\bfp_{s_1}/\sqrt{A_1}-\bfp_{\bar{s}_1}/\sqrt{\bar{A}_1}\big|/\bfp_s>2/\sqrt{N}\mid \Acal_{n-1}\Big)>\iota(N)>0~,}
	for some function $\iota(N)$ depending only on $N$, i.e., $\iota(N)$ is independent of $n$. Let $X_\pi^{(n)}$ and $\bar{X}_\pi^{(n)}$ be the \emph{score} of the $n$-th generation when the left-most child of the left-most node in the $(n-1)$-th generation are $(A_1,s_1)$ and $(\bar{A}_1,\bar{s}_1)$, respectively. Therefore, adding and subtracting the contribution of other nodes of generation $n$ to the scores, we have
	\eqan{\label{for:lem:subcriticality:04}
	&\prob\Big(\big|\bfp_{s_1}/\sqrt{A_1}-\bfp_{\bar{s}_1}/\sqrt{\bar{A}_1}\big|/\bfp_s>2/\sqrt{N}\mid \Acal_{n-1}\Big)\nn\\
	=&\prob\big( |{X}_\pi^{(n)}-\bar{X}_\pi^{(n)}|>2\mid\Acal_{n-1} \big)~.}
	By the triangle inequality and union bound, we upper bound the RHS of \eqref{for:lem:subcriticality:04} as
	\eqan{\label{for:lem:subcriticality:05}
	\prob\big( |{X}_\pi^{(n)}-\bar{X}_\pi^{(n)}|>2\mid\Acal_{n-1} \big)
	\leq& \prob \Big( |\bar{X}_\pi^{(n)}-{X}_\pi^{(n-1)}|>1~\text{or } |{X}_\pi^{(n)}-{X}_\pi^{(n-1)}|>1\mid \Acal_{n-1}\Big)\nn\\
	\leq& 2\prob \Big( |{X}_\pi^{(n)}-{X}_\pi^{(n-1)}|>1\mid\Acal_{n-1}\Big)~.}
	Assume that $\prob(S_k<\infty)>\gamma>0$ for some fixed $k\in\N$ and $\gamma>0.$ Then
	\eqan{\label{for:lem:subcriticality:06}
	\prob(T_k<\infty)\geq\sum\limits_{n>k}\prob(T_k=n+1\mid S_k=n)\prob(S_k=n)~.}
	Since $S_k$ is a stopping time, $\{S_k=n\}$ is $\Fcal_n$ measurable. Therefore, using \eqref{for:lem:subcriticality:04} and \eqref{for:lem:subcriticality:05},
	\eqn{\label{for:lem:subcriticality:07}
	\prob(T_k=n+1\mid S_k=n) = \prob(|X_\pi^{(n+1)}-X_\pi^{(n)}|>1\mid \Acal_n)>\iota(N)/2~.}
	Hence by \eqref{for:lem:subcriticality:07}, we obtain
	\eqan{\label{for:lem:subcriticality:08}
	\prob(T_k<\infty)>\iota(N)/2\sum\limits_{n>k}\prob(S_k=n)>\gamma \iota(N)/2~.}
	Since $\big(X_\pi^{(n)}\big)_{n\geq 1}$ converges almost surely, $\prob(T_k<\infty)\to0$ as $k\to\infty$. Therefore, $\prob(S_k<\infty)\to0$ as $k\to \infty$. Note that $\{S_k<\infty\}=\bigcup\limits_{n>k}\Acal_n$, and it is a decreasing sequence. Hence
	\eqn{\label{for:lem:subcriticality:09}
	\prob(\Acal_n~\text{i.o.})=\prob\Big(\bigcap\limits_{k\geq 1}\{S_k<\infty\}\Big)=\lim\limits_{k\to\infty}\prob(S_k<\infty)=0~.}
	Hence, $\Acal_n$ occurs for only finitely many $n$, proving that the age of the left-most node of the $n$-th generation diverges to $\infty$ almost surely. Therefore, upon truncating the nodes with age more than $1,~\PPT$ dies out almost surely.
\end{proof}

Next, in Proposition~\ref{prop:supercriticality}, we prove that the $\PPT$ survives when it is percolated with probability more than $1/r(\bar{\bfT}_k)$. This, along with Proposition~\ref{lem:subcriticality}, proves that $1/r(\bar{\bfT}_\kappa)$ is the critical percolation threshold of the P\'olya point tree. 

To prove Proposition~\ref{prop:supercriticality}, we truncate the $\PPT$ suitably and obtain the spectral norm of the offspring operator of the truncated $\PPT$. Then we show that upon percolating with probability $\pi$, the truncated tree has an ancestral line of nodes whose ages converge to $0$ with positive probability, proving that $\PPT$ also survives the percolation with positive probability.

\paragraph{Truncation.}
In the $\PPT$, consider the subtree where the age of the children of a node $(x,s)$ are restricted to the range $(0,bx]$ for some constant $b>1$. We call this truncated subtree the $b-\PPT$. Note that the mean-offspring kernel $\kappa$ changes in $b-\PPT$ to
\eqn{\label{eq:truncated-kernel}
\kappa_b((x,s),(y,t))=\kappa((x,s),(y,t))\one_{\{y\leq bx\}}~.}

Now we show that with positive probability there exists an infinite path from root such that the age of the nodes converges to $0$ in the $b-\PPT$. First we show that the integral operator $\bar{\bfT}_{\kappa_b}$ also has an eigenfunction corresponding to its spectral norm:

\begin{Lemma}[spectral norm of truncated operator]\label{lem:spectral-norm:truncated-operator}
	The spectral norm of the integral operator $\bar{\bfT}_{\kappa_b}$ is
	\eqn{\label{eq:lemma:spectral-radius}
	r(\bar{\bfT}_{\kappa_b}) = \frac{(1+q)c_{\old\old}+\sqrt{((1-q)c_{\old\old})^2+4qc_{\old\young}c_{\young\old}}}{2\chi-1}~,}
	where $q=1-b^{1/2-\chi}$. Furthermore, the eigenfunction of $\bar{\bfT}_{\kappa_b}$ corresponding to its spectral norm is given by
	\eqn{\label{eq:lemma:spectral-radius:efunc}
	h(x,s)=\frac{\bfu_s^b}{\sqrt{x}}~,}
	where $\bfu^b=(\bfu_{\old}^b,\bfu_{\young}^b)$ is the eigenvector of $\bfM_b$ corresponding to its largest eigenvalue and where $\bfM_b$ is defined as
	\[
		\bfM_b=\begin{bmatrix}
			c_{\old\old} & c_{\old\young}(1-b^{1/2-\chi})\\
			c_{\young\old} & c_{\young\young}(1-b^{1/2-\chi})
		\end{bmatrix}~.
	\]
\end{Lemma}
The proof of this lemma follows similarly as Theorem~\ref{thm:operator-norm:PPT}. So we defer the proof of this lemma to Section~\ref{sec:spectral_radius}. 

Note that $q\to1$ as $b\to\infty$, and consequently $r(\bar{\bfT}_{\kappa_b})\to r(\bar{\bfT}_{\kappa})$ as $b\to\infty$. Therefore, for any $\pi>1/r(\bar{\bfT}_{\kappa}),$ there exists a $b$ large enough such that $\pi r(\bar{\bfT}_{\kappa_b})>1~.$  We perform a change of measure to prove that the truncated tree survives percolation with positive probability. For this change of measure argument, we define a martingale similar to \eqref{eq:def:supermartingale} as
\eqn{\label{eq:def:martingale}
M^{\sss(n)}_b (x,s) :=\frac{1}{\rho_b^n}\sum\limits_{t\in\{ {\old,\young} \}} \int\limits_0^\infty Y_{(b)}^{(n)}((x,s),(y,t))\frac{h(y,t)}{h(x,s)}\,dy~,}
where $\rho_b=\pi r(\bar{\bfT}_{\kappa_b})$ and $Y_{(b)}^{(n)}((x,s),(y,t))$ puts a Dirac delta mass at the types of $n$-th generation offspring of $(x,s)$ in the percolated $b-\PPT$. Since $h$ is an eigenfunction of $\bar{\bfT}_{\kappa_b},~M^{\sss(n)}_b (x,s)$ is a martingale  for any choice of $(x,s)\in\Scal_e$. Since ${M}^{\sss(n)}_b (x,s)$ is a non-negative $L^1$ martingale, it converges almost surely to a non-negative martingale limit. We use the following Kesten-Stigum theorem for $b-\PPT$ to ensure that the martingale limit is non-zero with positive probability:
\begin{Theorem}[Kesten-Stigum Theorem for $b-\PPT$]\label{K-S-multi-type}
	$M_b^{(n)}(x,s)$ converges almost surely to a non-zero limit under the following assumption
	\eqn{\label{K-S-cond}
		\sup\limits_{(x,s)\in\Scal_e}\E\Big[ \left(M_b^{(1)}(x,s)\right)^{1+\eta} \Big]<\infty~\mbox{for some }\eta>0.}
\end{Theorem}
Next we prove that $M^{\sss(1)}_b (x,s)$ has finite $(1+\eta)$-th moment for some $\eta>0$ making the martingale uniformly integrable.
The following lemma provides us this uniform finite $(1+\eta)$-th moment for some $\eta>0$. Consequently, our martingale satisfies the $L\log L$ condition for Kesten-Stigum theorem, proved in \ref{K-S-multi-type} for $b-\PPT$ and hence $M^{\sss(n)}_b (x,s)$ converges to a positive limit with positive probability:
\begin{Lemma}[Finite $(1+\eta)$-th moment]
	\label{thm:finite:higher-moment}
	Fix any $\eta<2\chi-1$. Then
	\eqn{\label{eq:thm:finite:higher-moment}
		\sup\limits_{(x,s)\in \Scal} \E\Big[ \big(M_b^{(1)}(x,s)\big)^{1+\eta} \Big]<\infty~.}
\end{Lemma}

First we move on to prove Proposition~\ref{prop:supercriticality} subject to Theorem~\ref{K-S-multi-type} and Lemma~\ref{thm:finite:higher-moment} and then prove Theorem~\ref{K-S-multi-type} and Lemma~\ref{thm:finite:higher-moment}.
Let us denote $M_b(x,s)$ as the $n\to \infty$ limit of $M^{(n)}_b(x,s)$. Since $M^{(1)}_b(x,s)$ has a uniformly bounded $(1+\eta)$-th moment, it can also be shown that the limit $M_b(x,s)$ is positive with positive probability.

We next use a spine decomposition-type argument. Let $\prob_{(x,s)}$ denote the law of the $b-\PPT$ rooted at $(x,s)$. Now we look at $b-\PPT$ under the tilted measure
\eqn{\label{def:tilted-measure}
\,d\Qbb_{(x,s)}=M_b(x,s)\,d\prob_{(x,s)}~,}
Given a $b-\PPT$ rooted at $(x,s)$, we build a measure $\mu^\star$ on the set of all infinite paths $(x,s)=\bfx_0,\bfx_1,\ldots$ in the $b-\PPT$, such that, for any permissible sequence,
\eqn{\label{eq:mu:def}
\mu\big( \{\bfx_0,\bfx_1,\ldots\}:\bfx_0=(x,s),\ldots,\bfx_n=(y_n,t_n) \big)= \rho_b^{-n} \frac{M_b(\bfx_n)}{M_b(\bfx_0)}~.}
Let $(t_n)_{n\geq 0}$ be the sequence of labels in an infinite path $(x_n)_{n\geq 0}$. The next lemma shows that this label sequence thus generated is a stationary Markov chain with invariant distribution ${\bf{\upsilon}}^b=({\bf{\upsilon}}_\old^b,{\bf{\upsilon}}_\young^b)$ for some $({\bf{\upsilon}}_\old^b,{\bf{\upsilon}}_\young^b)$. The transition probability of the Markov chain is given as
\eqn{\label{eq:transition-prob:MC}
p_{s,t}=\frac{(\bfM_b)_{s,t}\bfu_t^b}{\lambda_{\bfM}^{\sss (b)}\bfu_s^b}\qquad\text{for }s,t\in\{\Old,\Young\}~.}
\begin{Lemma}[Markov property of labels]\label{lem:MC:label}
	The random sequence of labels of nodes in a path chosen from $\mu^\star$ is a stationary Markov chain with transition probabilities given by \eqref{eq:transition-prob:MC}.
\end{Lemma}
\begin{proof}
	For proving the Markov property, we look at the probability of $t_{n+1}$ conditionally on $\{t_0,t_1,\ldots,t_{n}\}$ and show that the probability depends only on $t_n$. Let $\bfa_{(n)}=\big(a^{(0)},a^{(1)},\ldots,a^{(n)}\big)\in\{\Old,\Young\}^{n}.$ Then, from the definition of conditional probability,
	\eqan{\label{for:lem:MC:01}
	\Qbb\big(t_{n+1}=\Old|(t_0,t_1,\ldots,t_n)=\bfa_{(n)}\big)=\frac{\Qbb\big( (t_0,t_1,\ldots,t_n,t_{n+1})=(\bfa_{(n)},\Old) \big)}{\Qbb\big((t_0,t_1,\ldots,t_n)=\bfa_{(n)}\big)}~.}
	Now the numerator of the RHS in \eqref{for:lem:MC:01} can be simplified as
	\eqan{\label{for:lem:MC:02}
	&\Qbb\big( (t_0,t_1,\ldots,t_n,t_{n+1})=(\bfa_{(n)},\Old) \big)\nn\\
	=&\rho_b^{-(n+1)}\int\limits_{(0,\infty)^{n}}\int_0^\infty \pi\kappa_b\big((x_n,a^{(n)}),(y,\Old)\big)\frac{h(y,\Old)}{h(x_n,a^{(n)})}\,dy\nn\\ &\qquad\qquad\times\prod\limits_{i=1}^{n}\pi\kappa_b\big((x_{i-1},a^{(i-1)}),(x_{i},a^{(i)})\big)\frac{h(x_i,a^{(i)})}{h(x_{i-1},a^{(i-1)})}\,dx_i\nn\\
	=&p_{a^{(n)},\old}\rho_b^{-n}\int\limits_{(0,\infty)^{n}}\prod\limits_{i=1}^{n}\pi\kappa_b\big((x_{i-1},a^{(i-1)}),(x_{i},a^{(i)})\big)\frac{h(x_i,a^{(i)})}{h(x_{i-1},a^{(i-1)})}\,dx_i\nn\\
	=&p_{a^{(n)},\old} \Qbb\big((t_0,t_1,\ldots,t_n)=\bfa_{(n)}\big)~.
	}
Since $\Qbb\big( (t_0,t_1,\ldots,t_n,t_{n+1})=(\bfa_{(n)},\Young) \big)=1-\Qbb\big( (t_0,t_1,\ldots,t_n,t_{n+1})=(\bfa_{(n)},\Old) \big),$ by \eqref{for:lem:MC:01} and \eqref{for:lem:MC:02} the lemma follows immediately.
\end{proof}

Now we investigate the age of the $n$-th node of the truncated P\'olya point tree in an infinite path sampled according to $\mu^\star$. These ages turn out to be multiplicative in nature:
\begin{Lemma}[Multiplicative property of the joint law of ages]\label{lem:multiplicative-property}
	Let $\{(X_0,t_0),(X_1,t_1),\ldots\}$ be an infinite path sampled from $\mu$. Then,
	\eqn{\label{eq:lem:multiplicative-property}
	(X_n)_{n\geq 0}\overset{d}{=}\Big(X_0\prod\limits_{i=1}^n R_i(t_i)\Big)_{n\geq 1}~,}
	for some independent random variables $(R_i(t))_{\{t\in\{\Old,\Young\},n\in\N\}}$.
\end{Lemma}
To prove Lemma~\ref{lem:multiplicative-property}, we compute the explicit distribution of the independent random variables, which also will be helpful for the last step of the proof for Proposition~\ref{prop:supercriticality}. First we calculate the distribution of ${X_n}/{X_{n-1}}$ conditionally on $X_{n-1},t_{n-1},t_n$ which turns out to be independent of $\big(X_{n-1},t_{n-1}\big)$. This proves Lemma~\ref{lem:multiplicative-property}.
\begin{proof}[Proof of Lemma~\ref{lem:multiplicative-property}]
	To prove the lemma, we calculate the joint distribution of ${X_n}/{X_{n-1}},t_n$ conditionally on $X_{n-1},t_{n-1}$. We compute for $a\in(0,1)$,
	\eqan{\label{for:lem:multiplicative-property:1}
	\Qbb\Big( \frac{X_n}{X_{n-1}}\geq a,t_n=\Old\mid X_{n-1}=x,t_{n-1} \Big)=&\rho_b^{-1}\int\limits_{ax}^x \pi\kappa_b((x,t_{n-1}),(y,\Old))\frac{h(y,\Old)}{h(x,t_{n-1})}\,dy\nn\\
	=&\Qbb(t_n=\Old\mid t_{n-1})(1-a^{\chi-\frac{1}{2}})~.
	}
	Similarly, for $a\in(1,b)$,
	\eqan{\label{for:lem:multiplicative-property:2}
		\Qbb\Big( \frac{X_n}{X_{n-1}}\geq a,t_n=\Young\mid X_{n-1}=x,t_{n-1} \Big)
		=&\rho_b^{-1}\int\limits_{ax}^{bx} \pi\kappa_b((x,t_{n-1}),(y,\Young))\frac{h(y,\Young)}{h(x,t_{n-1})}\,dy\nn\\
		=&\Qbb(t_n=\Young\mid t_{n-1})\frac{a^{\frac{1}{2}-\chi}-b^{\frac{1}{2}-\chi}}{1-b^{\frac{1}{2}-\chi}}~.
	}
	Therefore, by \eqref{for:lem:multiplicative-property:1} and \eqref{for:lem:multiplicative-property:2}, the tail distribution function of ${X_n}/{X_{n-1}}$ conditionally on $(X_{n-1},t_{n-1},t_n)$ does not depend on $(X_{n-1},t_{n-1})$ and is given by
	\eqan{\label{for:lem:multiplicative-property:3}
	\Qbb\Big( \frac{X_n}{X_{n-1}}\geq a\mid t_n=\Old,X_{n-1}=x,t_{n-1} \Big)&=\Qbb\Big( \frac{X_n}{X_{n-1}}\geq a\mid t_n=\Old \Big)\nn\\
	&\quad=1-a^{\chi-\frac{1}{2}},~\text{for }a\leq1~,\\
	\Qbb\Big( \frac{X_n}{X_{n-1}}\geq a\mid t_n=\Young,X_{n-1}=x,t_{n-1} \Big)&=\Qbb\Big( \frac{X_n}{X_{n-1}}\geq a\mid t_n=\Young \Big)\nn\\
	&\quad=\frac{a^{\frac{1}{2}-\chi}-b^{\frac{1}{2}-\chi}}{1-b^{\frac{1}{2}-\chi}}~\text{for }a\in(1,b)~.\nn
	}
	By \eqref{for:lem:multiplicative-property:3}, we can write $X_n$ in a multiplicative way as
	\eqn{\label{for:lem:multiplicative-property:4}
	X_n\overset{d}{=}X_{n-1}R_n(t_n)~,}
	where $R_n(t_n)$ has the distribution function defined in \eqref{for:lem:multiplicative-property:3} depending on $t_n$ being $\Old$ or $\Young$. Now iterating the same argument, we can write $X_n$ conditionally on $(t_0,\ldots,t_n)$ as
	\[
		X_n\overset{d}{=}X_0\prod\limits_{i=1}^n R_i(t_i)~,
	\]
	where $\big(R_i(t_i)\big)_{i\geq 1}$ are (conditionally) independent random variables. Hence, $X_n$ has the desired multiplicative structure.
\end{proof}
Now we have all the required tools to prove Proposition~\ref{prop:supercriticality}. Before giving the details we give an outline of the remaining proof. We prove that $\Qbb$-almost surely, $\mu$-almost every path has the property that
\eqn{\label{for:prop:supercriticality:1}
\lim\limits_{n\to\infty}\frac{\log(X_n)}{n}= c<0~.}
This implies that $\Qbb$-almost surely, the log values of the ages of nodes in an infinite path converge to $-\infty$. Therefore, there exists an ancestral line of nodes whose age converges to $0$. For the $b-\PPT$ rooted at some $(x,s)\in\Scal$, it will follow that there exists an ancestral line of particles whose ages converge to $0$ with positive probability. Therefore, $b-\PPT$ survives the percolation with positive probability.
\begin{proof}[Proof of Proposition~\ref{prop:supercriticality}]
	For proving \eqref{for:prop:supercriticality:1}, we use Lemma~\ref{lem:multiplicative-property} as
	\eqn{\label{for:prop:supercriticality:2}
	\frac{\log(X_n)}{n}\overset{d}{=}\frac{\log(X_0)+\sum\limits_{i=1}^{n_\old} \log(R_i(\Old))+\sum\limits_{j=1}^{n_\young} \log(R_j(\Young))}{n}~,}
	where $n_\old$ and $n_\young$ are the number of $\Old$ and $\Young$ labels in $\{t_1,t_2,\ldots,t_n\}$ respectively. Since $R_i(\Old)$ and $R_j(\Young)$ are independent random variables, and $\{t_0,t_1,\ldots\}$ is a stationary Markov chain with limiting distribution ${\bf{\upsilon}}^b=({\bf{\upsilon}}_\old^b,{\bf{\upsilon}}_\young^b)$, by the strong law of large numbers
	\eqan{\label{for:prop:supercriticality:3}
	\frac{\log(X_n)}{n}\overset{a.s.}{\to}{\bf{\upsilon}}_\old^b \E[\log(R_1(\Old))]+{\bf{\upsilon}}_\young^b \E[\log(R_1(\Young))]~,}
	where ${\bf{\upsilon}}^b=({\bf{\upsilon}}_\old^b,{\bf{\upsilon}}_\young^b)$ is the stationary distribution of $\{t_0,t_1,\ldots\}$. By \eqref{for:lem:multiplicative-property:3}, we compute
	\eqan{
		&\E[\log(R_1(\Old))]=-\frac{1}{\chi-\frac{1}{2}},\label{for:prop:supercriticality:4-1}\\
		\text{and}\qquad&\E[\log(R_1(\Young))]=\frac{1}{1-b^{1/2-\chi}}\Big[ \frac{1-b^{1/2-\chi}}{\chi-\frac{1}{2}}-b^{1/2-\chi}\log(b) \Big]~.\label{for:prop:supercriticality:4-2}
	}
	Next we show that ${\bf{\upsilon}}_\old^b>{\bf{\upsilon}}_\young^b$. For proving this, we make use of a small computation that holds true in general for a $2$-state Markov process.
	Let 
	\eqn{\label{for:clarity:matrix-1}
	A=\begin{bmatrix}
		e_1 & 1-e_1\\
		e_2 & 1-e_2
	\end{bmatrix}}
	be a transition matrix and $\theta=(\theta_1,\theta_2)$ be its stationary distribution. Then,
	\eqn{\label{for:clarity:matrix-2}
	\frac{\theta_1}{\theta_2}=\frac{e_2}{1-e_1}~.}
	Therefore, necessary and sufficient condition for $\theta_1>\theta_2$ is $e_1+e_2>1$.
	Hence, ${\bf{\upsilon}}_\old^b>{\bf{\upsilon}}_\young^b$ holds if and only if $p_{\old\old}+p_{\young\old}>1$ holds. Substituting the values of $p_{\old\old}$ and $p_{\young\old}$ from \eqref{eq:transition-prob:MC}, we have
	\eqn{\label{for:prop:supercriticality:5}
	p_{\old\old}+p_{\young\old}=\frac{c_{\old\old}}{\lambda_{\sss \bfM}^{(b)}}+\frac{c_{\young\old}\bfu_{\old}^b}{\lambda_{\sss \bfM}^{(b)}\bfu_{\young}^b}=\frac{c_{\young\young}}{\lambda_{\sss \bfM}^{(b)}}+\frac{c_{\young\old}\bfu_{\old}^b}{\lambda_{\sss \bfM}^{(b)}\bfu_{\young}^b}>p_{\young\young}+p_{\young\old}=1~.}
	Therefore, by \eqref{for:prop:supercriticality:3}-\eqref{for:prop:supercriticality:5}, we have that ${\log(X_n)}/{n}\to c$ for some $c<0$.
	For $b\to\infty,~\upsilon^{b}\to(1/2,1/2)$ and the second factor in the RHS of \eqref{for:prop:supercriticality:4-2} also converges to $0$, which makes the RHS of \eqref{for:prop:supercriticality:3}, $0$.
	As a result, $\Qbb_{(x,s)}$ almost surely, the ages of the nodes in $\PPT$ converges to $0.~\Qbb_{(x,s)}$ is a tilted measure of $\prob_{(x,s)}$ with the martingale limit of $M_b^{(n)}(x,s)$ being the Radon-Nikodym derivative. Since $M_b^{(n)}(x,s)$ has positive limit with positive probability, age of the nodes in the $n$-th generation of the $b-\PPT$ converges to $0$ with positive probability. Since $b-\PPT$ is a truncated $\PPT,~\PPT$ also survives with positive probability.
\end{proof}
Now, we remain to prove the Kesten-Stigum theorem for the $b-\PPT$ of Theorem~\ref{K-S-multi-type}. First, we prove Theorem~\ref{K-S-multi-type} and then prove Lemma~\ref{thm:finite:higher-moment} to satisfy the condition for Theorem~\ref{K-S-multi-type}. 

\begin{proof}[Proof of Theorem~\ref{K-S-multi-type}]
We prove this theorem following the same line of proof done by Athreya in \cite{A2000}, adapted for our martingale as done in \cite{KS01}. For the sake of simplicity of the notations we use $W_n(y)$ for $y\in\Scal_e$, in the place of $M_b^{(n)}(x,s)$ for $(x,s)\in\Scal_e$ and $W(y)$ as the $n\to\infty$ limit of $W_n(y)$.
	
Let $\Mcal\equiv\{ \mu:\mu(\cdot)=\sum_{i=1}^n \delta_{x_i}(\cdot)~\text{for some }n<\infty,~x_1,x_2,\ldots,x_n\in\Scal_e \}$ where $\delta_x(\cdot)$ is the delta measure at $x$, i.e., $\delta_x(A)=1$ if $x\in A$ and $0$ if $x\notin A$. Clearly, $\Mcal$ is closed under addition.
Let $z_n$ denote the number of particles in the $n$-th generation of the tree, and for any $x\in\Scal_e,$ let $\xi^x$ denote the empirical measures of the children of $x$ in $\Scal_e$. Let $\{x_{ni}:i\in [z_n]\}$ denote the particles in the tree in $n$-th generation of the tree. Then adapting the notions of \cite{A2000}, we define $Z_n$ as
\eqn{\label{eq:K-S:1}
		Z_n=\sum\limits_{i=1}^{z_n} \xi^{x_{ni}}~,}
where $\xi^{x_{ni}}$ are independent $\Mcal$-valued random variable, and denote $V(\mu)$ as
\[
	V(\mu)=\int h\,d\mu~,
\]
where $h$ is the eigenfunction defined in \eqref{eq:lemma:spectral-radius:efunc}. 
	For any initial value $z\in \Mcal,$
	\eqn{\label{eq:K-S:02}
		\Qbb(z,\,d\mu)=\frac{V(\mu)\prob(z,\,d\mu)}{\rho_b V(z)}~.} 
	Choosing $z=\delta_{x_0},$ we obtain the same $\Qbb_{x_0}$ measure as defined in \eqref{def:tilted-measure} for some $x_0\in\Scal_e$.
	From Corollary~$1$ of \cite{A2000}, for $x_0\in\Scal_e$,
	\eqn{\label{eq:K-S:021}
		\E_{x_0}[W(x_0)]=1~\text{under}~\prob_{x_0},~\text{if and only if}~\Qbb_{x_0}(W(x_0)=\infty)=0~.}
	The random variable $\tilde{\xi}^x$ has a size-biased distribution of $\xi^x$ in that
	\eqn{\label{eq:K-S:2}
		\prob\big( \tilde{\xi}^x\in \,dm \big)=\frac{V(m)\prob\big( {\xi}^x\in \,dm \big)}{\rho_b h(x)}~.}
	Under $\Qbb$ measure, $b-\PPT$ has an infinite line of nodes. This infinite line of nodes is termed as the `spine' of the tree. Offspring distribution of the nodes of spine is the size-biased distribution defined in \eqref{eq:K-S:2}. Rest of the nodes in the tree has usual offspring distribution. Following the calculation in \cite[page~329-330]{A2000}, it can be shown that the markov chain $(Z_n)_{n=0}^\infty$ with transition function $\Qbb$ defined in \eqref{eq:K-S:02}, evolves in the following manner. Given $Z_n=(x_{n1},x_{n2},\ldots,x_{nz_n}),~Z_{n+1}$ is generated as follows:
	\begin{enumerate}
		\item Choose individual $x_{ni}$ with probability $\frac{h(x_{ni})\one_{[x_{ni}\sim x_{(n-1)j^\star}]}}{V(\xi^{x_{(n-1)j^\star}})}$ and name it $x_{nj^\star}$;
		\vspace{3pt}
		\item Choose its offspring $\tilde{\xi}^{x_{nj^\star}}$ according to the one in \eqref{eq:K-S:2};
		\vspace{3pt} 
		\item For all other individuals in the generation, choose the offspring process $\xi^{x_{ni}}$ according to the original probability distribution $\prob(\xi^{x_{ni}}\in\,dm)$;
		\vspace{3pt}
		\item $\tilde{\xi}^{x_{nj^\star}}$ and $\xi^{x_{ni}}$ are all chosen independently for $i\neq j^\star$;
		\vspace{3pt}
		\item Set $Z_{n+1}=\tilde{\xi}^{x_{nj^\star}}+\sum\limits_{i\neq j^\star}\xi^{x_{ni}}~.$
	\end{enumerate}
	Therefore, $\{x_{nj^\star},n=1,2,\ldots\}$ keeps track of the `spine' component of the tree. Next we follow the line of proof in \cite[equation~$(15),(19a),(19b)$]{A2000} with the adapted martingale $W_n(x_0)$ to obtain
	\eqan{\label{eq:K-S:3}
		\E_{\Qbb}\Big[ W_{n+1}(x_0)\mid \tilde{\xi}^{x_{1j^\star}},\tilde{\xi}^{x_{2j^\star}},\ldots \Big] \leq& 1+\sum\limits_{r=0}^\infty \frac{V(\tilde{\xi}^{x_{rj^\star}})}{\rho_b^{r+1}h(x_0)}\nonumber\\
		=&1+\sum\limits_{r=0}^\infty \left(\frac{V(\tilde{\xi}^{x_{rj^\star}})}{\rho_b^{r+1}h(x_{rj^\star})}\right)\left(\frac{h(x_{rj^\star})}{h(x_0)}\right)\equiv W^\star(x_0)~.}
	Now we prove that there exists $\omega\in(0,1)$ such that $$\Dcal_r=\left\{\left(\frac{V(\tilde{\xi}^{x_{rj^\star}})}{\rho_b^{r+1}h(x_{rj^\star})}\right)\left(\frac{h(x_{rj^\star})}{h(x_0)}\right)>\omega^r\right\}$$
	happens at most finitely often, which in turn proves that $W^\star(x_0)$ is finite $\Qbb$ almost surely.	
	For proving that $\Dcal_r$ happens finitely often, define the following:
	\eqan{
		\Bcal_r=&\left\{ \frac{V(\tilde{\xi}^{x_{rj^\star}})}{\rho_b h(x_{rj^\star})}>(\rho_b\omega)^{r(1-\alpha)}  \right\}~,\label{eq:K-S:4}\\
		\text{and}\qquad\Ccal_r=& \left\{ \frac{h(x_{rj^\star})}{h(x_0)}>(\rho_b\omega)^{r\alpha} \right\}~,\label{eq:K-S:5}}
	for some $\alpha\in(0,1)$. Then by the union bound,
	\eqn{\label{eq:K-S:6}\Qbb(\Dcal_r~\text{i.o.})\leq \Qbb(\Bcal_r~\text{i.o.})+\Qbb(\Ccal_r~\text{i.o.})~.}
	The author has proved in \cite{A2000} that a sufficient condition to establish $\Qbb(\Bcal_r~\text{i.o.})=0$ is the finiteness of $\int\limits_1^\infty \bar{f}(\e^t)\,dt$, where 
	\eqn{\label{eq:K-S:07}
		\bar{f}(t)\equiv\sup\limits_x\prob\left(\frac{V(\tilde{\xi}^x)}{\rho_b h(x)}>t\right)~.}
	Although the proof in \cite{A2000} is there for a different denominator, but the argument here still follows exactly the same way as in \cite{A2000}.
	The following claim provides a sufficient condition for the finiteness of $\int\limits_1^\infty \bar{f}(\e^t)\,dt$:
	\begin{Claim}[A sufficient condition]\label{lem:equivalent:condition}
		$\int\limits_1^\infty \bar{f}(\e^t)\,dt$ is finite if for some $\eta>0$,
		\eqn{\label{eq:lem:equivalence}
			\sup\limits_{(x,s)\in\Scal_e}\E\Big[ M_b^{(1)}(x,s)^{1+\eta} \Big]<\infty~,}
		where $f$ is as defined in \eqref{eq:K-S:07}.
	\end{Claim}
	We defer proof of this claim to the appendix~\ref{app:sufficiency}.
	By Lemma~\ref{lem:equivalent:condition}, finiteness of  $\sup\limits_x\E\left[ \left(\frac{V({\xi}^x)}{\rho_b h(x)}\right)^{1+\eta}\right]$ implies finiteness of $\int\limits_1^\infty \bar{f}(\e^t)\,dt$, and Lemma~\ref{thm:finite:higher-moment} provides us with the uniform bound on the $1+\eta$-th moment of $W_1(y)$. Therefore, 
	\eqn{\label{eq:K-S:7}
		\Qbb(\Bcal_r~\text{i.o.})=0~.}
	For the analysis of $\Ccal_r$, we use the properties of $b-\PPT$ under $\Qbb$, more specifically the proof of Proposition~\ref{prop:supercriticality}.
	Let $x_{rj^\star}$ has label $t_r\in\{\Old,\Young\}$ and age $a_{rj^\star}\in (0,\infty)$. Therefore, $\log h(x_{rj^\star})=\log(\bfu_{t_r}^b)+\log(a_{rj^\star})/2$, where $\bfu^b=(\bfu_\old^b,\bfu_\young^b)$ is as defined in Lemma~\ref{lem:spectral-norm:truncated-operator}.
	By \eqref{for:prop:supercriticality:2}-\eqref{for:prop:supercriticality:4-2}, we can prove that for any fixed $b>0$,
	\eqn{\label{eq:K-S:8}
		\frac{\log(h(x_{rj^\star}))-\log h(x_0)}{r}\overset{a.s.}{\to} \frac{{\bf{\upsilon}}_\old^b-{\bf{\upsilon}}_\young^b}{2\chi-1} + \frac{b^{1/2-\chi}\log(b){\bf{\upsilon}}_\young^b}{2(1-b^{1/2-\chi})}~.}
	Now using \eqref{for:clarity:matrix-1}-\eqref{for:clarity:matrix-2}, the stationary measure ${\bf{\upsilon}}^b=({\bf{\upsilon}}_\old^b,{\bf{\upsilon}}_\young^b)$ can be explicitly calculated as
	\[
	{\bf{\upsilon}}_\old^b=\frac{\frac{c_{\young\old}\bfu_{\old}^b}{\lambda_{\bfM}^{\sss(b)}\bfu_{\young}^b}}{\frac{c_{\young\old}\bfu_{\old}^b}{\lambda_{\bfM}^{\sss(b)}\bfu_{\young}^b}+\frac{c_{\old\young}(1-b^{1/2-\chi})\bfu_{\young}^b}{\lambda_{\bfM}^{\sss(b)}\bfu_{\old}^b}}\quad\text{and}\quad {\bf{\upsilon}}_\young^b=1-{\bf{\upsilon}}_\old^b~,
	\]
	where $\lambda_{\bfM}^{(b)}$ is as defined in \eqref{for:lem:spectra:truncated-operator:2} and $\bfu^b=(\bfu_\old^b,\bfu_{\young}^b)$ is the right eigenvector of $\bfM_b$ defined in Lemma~\ref{lem:spectral-norm:truncated-operator}. For any $b>0~,\bfu^b$ can be explicitly computed as
	\[
	\bfu_{\old}^b=\frac{c_{\old\young}(1-b^{1/2-\chi})}{\lambda_{\bfM}^{(b)}+c_{\old\young}(1-b^{1/2-\chi})-c_{\old\old}}\quad\text{and}\quad\bfu_{\young}^b=1-\bfu_{\old}^b~.
	\]
	Note that all of  $\{\bfu_\old^b,\bfu_{\young}^b,{\bf{\upsilon}}_\old^b,{\bf{\upsilon}}_\young^b\}$ are continuous in $b,$ and as $b\to \infty,~{\bf{\upsilon}}^b\to (1/2,1/2)$, while the RHS of \eqref{eq:K-S:8} decreases to $0$. 
	
	On the other hand, as $b\to\infty,~\rho_b$ increases to $\rho=\pi r(\bfT_{\kappa}),$ which is strictly greater than $1$. Therefore, there exists a $b$ large enough and suitable $\omega\in(1/\rho,1)$ such that $\log(\rho_b\omega)>\log(\rho\omega)/2$.
	Similarly, for $b$ large enough, RHS of \eqref{eq:K-S:8} can be shown to be less than $\log(\rho\omega)/4$.
	Hence $\Ccal_r$ occurs finitely often, and by \eqref{eq:K-S:6}, $\Qbb(\Dcal_r~\text{i.o.})=0$ completes the argument that $\Dcal_r$ happens finitely often $\Qbb_{x_0}$ almost surely.	
	Following the argument in \cite[Proof of Theorem~3]{A2000}, we prove 
	\eqn{\label{eq:K-S:04}
		\Qbb_{x_0}(W(x_0)<\infty)=1\quad\text{or}\quad\Qbb_{x_0}(W(x_0)=\infty)=0~.}
	Therefore, by \eqref{eq:K-S:021}, we conclude
	\eqn{\label{eq:K-S:041}
		\E_{x_0}[W(x_0)]=1\quad\text{under}\quad\prob_{x_0}~.}
	Hence, $W(x_0)$ is non-zero random variable $\prob_{x_0}$ almost surely. 
\end{proof}

Now we remain to prove Lemma~\ref{thm:finite:higher-moment} to satisfy the sufficiency condition in Lemma~\ref{lem:equivalent:condition}. For any $(x,s)$ in the $b-\PPT$, we assign two bins $B_\old(x,s),B_\young(x,s)\subset [0,bx]\times\{\Old,\Young\}$ such that all the $\Old$ labeled children of $(x,s)$ are in $B_\old(x,s)$. Similarly, all the $\Young$ labeled children of $(x,s)$ having age in $[x,bx]$ are in $B_\young(x,s)$. From the construction of the P\'olya point tree, observe that only finitely many offspring of $(x,s)$ are in $B_t(x,s)$ for both $t\in\{\Old,\Young\}$. Let $n_t$ denote the (possibly random) number of offspring of $(x,s)$ in bin $B_t(x,s),$ and $(A_{ti})_{i\in[n_t]}$ denote the types of offspring of $(x,s)$ in $B_t(x,s)$ for $t\in\{\Old,\Young\}$.

We write $M^{(1)}_b(x,s)$ as
\eqn{\label{eqn:martingale:independent-rv}
	M^{(1)}_b (x,s)= \frac{1}{\rho_b^n}\sum\limits_{t\in\{\old,\young\}}\sum\limits_{i=1}^{n_t} \frac{h(A_{ti})}{h(x,s)}~.}
This representation helps us to prove Lemma~\ref{thm:finite:higher-moment}, proving the finiteness of $(1+\eta)$-th moment of $M^{(1)}_b (x,s)$ for some $\eta>0$:
\begin{proof}[Proof of Lemma~\ref{thm:finite:higher-moment}]
	Fix any $(x,s)\in\Scal$. Then using \eqref{eqn:martingale:independent-rv}, $\big(M^{(1)}_b(x,s)\big)^{1+\eta}$ can be rewritten as
	\eqan{\label{for:lem:finite:higher-moment:1}
		\big(M^{(1)}_b(x,s)\big)^{1+\eta}
		=& (\rho_b h(x,s))^{-(1+\eta)} \Big( \sum\limits_{t\in\{\old,\young\}}\sum\limits_{i=1}^{n_t} {h(A_{ti})}\Big)^{1+\eta}\nn\\
		=&(\rho_b h(x,s))^{-(1+\eta)} \sum\limits_{t\in\{\old,\young\}}\sum\limits_{i=1}^{n_t} h(A_{ti}) \Big(\sum\limits_{l\in\{\old,\young\}}\sum\limits_{j=1}^{n_l} h(A_{lj})\Big)^\eta\\
		\leq&(\rho_b h(x,s))^{-(1+\eta)} \sum\limits_{t\in\{\old,\young\}}\sum\limits_{i=1}^{n_t} h(A_{ti})\Big[ \Big( \sum\limits_{j=1}^{n_\old} h(A_{\old j}) \Big)^\eta + \Big( \sum\limits_{j=1}^{n_\young} h(A_{\young j}) \Big)^\eta \Big]\nonumber~.
		}
	The inequality in \eqref{for:lem:finite:higher-moment:1} follows from the fact that $(a+b)^\eta\leq a^\eta+b^\eta$ for all $\eta<1$ and $a,b>0$.
	Since $\rho_b$ and $h(x,s)$ are constants, we leave out those terms in \eqref{for:lem:finite:higher-moment:1} for the time being and concentrate on the random part. More precisely, instead of looking at $\E\Big[ \big(M^{(1)}_b(x,s)\big)^{1+\eta} \Big]$, we upper bound
	\eqan{\label{for:lem:finite:higher-moment:2}
	&(\rho_b h(x,s))^{1+\eta}\E\Big[ \big(M^{(1)}_b(x,s)\big)^{1+\eta} \Big]\nn\\
	\leq& ~\E\Big[\sum\limits_{t\in\{\old,\young\}}\sum\limits_{i=1}^{n_t} h(A_{ti})\Big(\sum\limits_{i=1}^{n_t} h(A_{ti})\Big)^{\eta}\Big]+\E\Big[\sum\limits_{i=1}^{n_\old} h(A_{\old i})\Big]\E\Big[ \Big( \sum\limits_{j=1}^{n_\young} h(A_{\young j}) \Big)^\eta \Big]\nn\\
	&\hspace{5cm}+\E\Big[\sum\limits_{i=1}^{n_\young} h(A_{\young i})\Big]\E\Big[ \Big( \sum\limits_{j=1}^{n_\old} h(A_{\old j}) \Big)^\eta \Big]~.
	}
	By the construction of the $b-\PPT$, the ages of the $\Old$ and $\Young$ labeled children of $(x,s)$ are independent. Therefore, the expectations in the second and third term of the RHS of \eqref{for:lem:finite:higher-moment:2} splits. We deal with the three sums in \eqref{for:lem:finite:higher-moment:2} separately.
	\paragraph{Upper bounding the first sum in \eqref{for:lem:finite:higher-moment:2}.} 
	First we use a similar inequality, splitting between $j=i$ and $j\neq i$ to get 
	\eqan{\label{for:lem:finite:higher-moment:02}
		&\E\Big[ \sum\limits_{t\in\{\old,\young\}}\sum\limits_{i=1}^{n_t} h(A_{ti})\Big( \sum\limits_{j=1}^{n_t} h(A_{tj}) \Big)^\eta  \Big]\nn\\
		\leq& \E\Big[ \sum\limits_{t\in\{\old,\young\}} \sum\limits_{i=1}^{n_t} h(A_{ti})^{1+\eta}\Big] + \sum\limits_{t\in\{\old,\young\}}\E\Big[ \sum\limits_{i=1}^{n_t} h(A_{ti})\Big( \sum\limits_{\substack{j=1\\j\neq i}}^{n_t} h(A_{tj}) \Big)^\eta   \Big]~.}
	Note that $(A_{\old i})_{i\in[n_\old]}$ are independent by the construction of the P\'olya point tree. Therefore, for $k=\Old$, the second expectation in \eqref{for:lem:finite:higher-moment:02} can be simplified as
	\eqan{\label{for:lem:finite:higher-moment:03}
		\E\Big[ \sum\limits_{i=1}^{n_\old} h(A_{\old i})\Big( \sum\limits_{\substack{j=1\\j\neq i}}^{n_\old} h(A_{\old j}) \Big)^\eta   \Big] =& \sum\limits_{i=1}^{n_\old}\E\Big[  h(A_{\old i})\Big]\E \Big[\Big( \sum\limits_{\substack{j=1\\j\neq i}}^{n_\old} h(A_{\old j}) \Big)^\eta   \Big]  \nn\\
		\leq&\E\Big[ \sum\limits_{i=1}^{n_\old} h(A_{\old i})\Big]\E \Big[\Big( \sum\limits_{j=1}^{n_\old}h(A_{\old j}) \Big)^\eta   \Big] ~.
	}
	Since $0<\eta<1$, we use Jensen's inequality for concave functions to obtain
	\eqn{\label{for:lem:finite:higher-moment:04}
		\E\Big[ \sum\limits_{i=1}^{n_\old} h(A_{\old i})\Big( \sum\limits_{\substack{j=1\\j\neq i}}^{n_\old} h(A_{\old j}) \Big)^\eta   \Big] \leq \E\Big[ \sum\limits_{i=1}^{n_\old} h(A_{\old i})\Big]^{1+\eta}~.}
	All of $(A_{\young i})_{i\in[n_\young]}$ have label $\Young,$ and the ages are the occurrence times given by a mixed-Poisson process on $[x,bx]$ with random intensity $\rho_{(x,s)}(y)$ as defined in \eqref{for:pointgraph:poisson}. Therefore, $n_\young$ is a mixed-Poisson random variable with intensity $\Gamma_{(x,s)}(b^{1-\chi}-1)$. Conditionally on $n_\young,$ the ages of $ (A_{\young i})_{i\in[n_\young]}$ are i.i.d.\ random variables independent of $n_\young$. Let $A_\young$ denote the distribution of $A_{\young i}$ conditionally on $n_\young$. Then, for $t=\Young$, the second expectation in \eqref{for:lem:finite:higher-moment:02} can be simplified as
	\eqan{\label{for:lem:finite:higher-moment:05}
		\E\Big[ \sum\limits_{i=1}^{n_\young} h(A_{\young i})\Big( \sum\limits_{\substack{j=1\\j\neq i}}^{n_\young} h(A_{\young j}) \Big)^\eta   \Big]=&\E\Big[ \sum\limits_{i=1}^{n_\young}\E_{n_\young}\Big[h(A_{\young i})\Big( \sum\limits_{\substack{j=1\\j\neq i}}^{n_\young} h(A_{\young j}) \Big)^\eta \Big]\Big]\nn\\
		=&  \E\Big[ n_\young \E_{n_\young}[h(A_{\young 1})]\E_{n_\young}\Big[ \Big( \sum\limits_{j=1}^{n_\young -1} h(A_{\young j}) \Big)^\eta\Big]\Big]~,}
	where $\E_{n_\young}[\cdot]=\E[\cdot\mid{n_\young}]$. Again by Jensen's inequality for the concave function $x\mapsto x^\eta$, we bound \eqref{for:lem:finite:higher-moment:05} by
	\eqan{\label{for:lem:finite:higher-moment:06}
		\E\Big[ n_\young \E_{n_\young}[h(A_{\young1})]\E_{n_\young}\Big[ \Big( \sum\limits_{j=1}^{n_\young-1} h(A_{\young j}) \Big)^\eta\Big]\Big]\leq& \E\Big[ n_\young \E[h(A_{\young})]\Big( \sum\limits_{j=1}^{n_\young-1} \E_{n_\young}[h(A_{\young j})] \Big)^\eta\Big]\nn\\
		=&\E\Big[ n_\young(n_\young-1)^{\eta} \Big]\E[h(A_{\young})]^{1+\eta} ~.}
	Using Holder's inequality, the first expectation in \eqref{for:lem:finite:higher-moment:06} can be upper bounded as
	\eqn{\label{for:lem:finite:higher-moment:07}
		\E\Big[ n_\young(n_\young-1)^{\eta} \Big]\leq \E[n_\young]^{1-\eta}\E[n_\young(n_\young-1)]^{\eta}=\E[n_\young]^{1-\eta}\E[\Gamma_{(x,s)}^2]^{\eta}(b^{1-\chi}-1)^{2\eta}~.}
	The last equality in \eqref{for:lem:finite:higher-moment:07} follows from the fact that $n_\young$  is a mixed-Poisson random variable with intensity $\Gamma_{(x,s)}(b^{1-\chi}-1)$. Note that $\Gamma_{(x,s)}$ is a Gamma random variable with parameters $m+\delta+\one_{\{s=\young\}}$ and $1$. Since $m\geq 2$ and $\delta>0$, it can be shown that $\E[\Gamma_{(x,s)}^2]\leq 2\E[\Gamma_{(x,s)}]^2$. Therefore, \eqref{for:lem:finite:higher-moment:07} can be further upper bounded by
	\eqn{\label{for:lem:finite:higher-moment:08}
		\E\Big[ n_\young(n_\young-1)^{\eta} \Big]\leq 2^\eta\E[n_\young]^{1-\eta}\Big(\E[\Gamma_{(x,s)}](b^{1-\chi}-1)\Big)^{2\eta}=2^{\eta}\E[n_\young]^{1+\eta}(b^{1-\chi}-1)^{2\eta}~.}
	By \eqref{for:lem:finite:higher-moment:08}, the LHS of \eqref{for:lem:finite:higher-moment:05} is upper bounded by
	\eqn{\label{for:lem:finite:higher-moment:09}
		\text{LHS of \eqref{for:lem:finite:higher-moment:05}}\leq 2^\eta \Big( \E[n_\young]\E[h(A_{\young})] \Big)^{1+\eta}(b^{1-\chi}-1)^{2\eta}~.}
	Recall that, conditionally on $n_\young,$ the ages of $(A_{\young i})_{i\in[n_\young]}$ are i.i.d.\ random variables, independent of $n_\young$. Let, further, $A_\young$ denote the distribution of $A_{\young i}$ conditionally on $n_\young$. Then, by Wald's identity
	\eqn{\label{for:lem:finite:higher-moment:10}
		\E\Big[ \sum\limits_{i=1}^{n_\young} h(A_{\young i}) \Big]=\E[n_\young ]\E[h(A_{\young})] ~.
	}
	Therefore, by \eqref{for:lem:finite:higher-moment:06}, \eqref{for:lem:finite:higher-moment:09} and \eqref{for:lem:finite:higher-moment:10}, the LHS of \eqref{for:lem:finite:higher-moment:05} can be upper bounded by 
	\eqn{\label{for:lem:finite:higher-moment:11}
		\E\Big[ \sum\limits_{i=1}^{n_\young} h(A_{\young i})\Big( \sum\limits_{\substack{j=1\\j\neq i}}^{n_\young} h(A_{\young j}) \Big)^\eta   \Big] \leq 2^\eta \E\Big[ \sum\limits_{i=1}^{n_\young} h(A_{\young i}) \Big]^{1+\eta}(b^{1-\chi}-1)^{2\eta}~.}
	Now, by \eqref{for:lem:finite:higher-moment:04} and \eqref{for:lem:finite:higher-moment:11}, we upper bound the second term in the RHS of \eqref{for:lem:finite:higher-moment:02} by
	\eqan{\label{for:lem:finite:higher-moment:12}
		\sum\limits_{t\in\{\old,\young\}}^\infty\E\Big[ \sum\limits_{i=1}^{n_t} h(A_{ti})\Big( \sum\limits_{\substack{j=1\\j\neq i}}^{n_t} h(A_{tj}) \Big)^\eta   \Big] &\leq 2^\eta (b^{1-\chi}-1)^{2\eta}  \sum\limits_{t\in\{\old,\young\}}\E\Big[ \sum\limits_{i=1}^{n_t} h(A_{ti}) \Big]^{1+\eta}\\
		&\leq 2^\eta(b^{1-\chi}-1)^{2\eta} \E\Big[ \sum\limits_{t\in\{\old,\young\}}\sum\limits_{i=1}^{n_t} h(A_{ti}) \Big]^{1+\eta}~.\nn}
	Therefore, by \eqref{for:lem:finite:higher-moment:12}, the LHS of \eqref{for:lem:finite:higher-moment:02} is further upper bounded as
	\eqan{\label{for:lem:finite:higher-moment:13}
		&\E\Big[ \sum\limits_{t\in\{\old,\young\}}\sum\limits_{i=1}^{n_t} h(A_{ti})\Big( \sum\limits_{j=1}^{n_t} h(A_{tj}) \Big)^\eta \Big]\nn\\
		&\leq \E\Big[ \sum\limits_{t\in\{\old,\young\}} \sum\limits_{i=1}^{n_t} h(A_{ti})^{1+\eta}\Big] + 2^\eta (b^{1-\chi}-1)^{2\eta} \E\Big[ \sum\limits_{t\in\{\old,\young\}}\sum\limits_{i=1}^{n_t} h(A_{ti}) \Big]^{1+\eta}~.
	}
	\paragraph{Upper bounding the second and third sum in \eqref{for:lem:finite:higher-moment:2}.}
	Now we move on to upper bounding the second and third sum in \eqref{for:lem:finite:higher-moment:2}. By Jensen's inequality,
	\eqan{\label{for:lem:finite:higher-moment:14}
	&\E\Big[\sum\limits_{i=1}^{n_\old} h(A_{\old i})\Big]\E\Big[ \Big( \sum\limits_{j=1}^{n_\young} h(A_{\young j}) \Big)^\eta \Big]+\E\Big[\sum\limits_{i=1}^{n_\young} h(A_{\young i})\Big]\E\Big[ \Big( \sum\limits_{j=1}^{n_\old} h(A_{\old j}) \Big)^\eta \Big]\nn\\
	&\qquad\leq \E\Big[ \sum\limits_{t\in\{\old,\young\}}\sum\limits_{i=1}^{n_t} h(A_{ti}) \Big] \E\Big[ \Big( \sum\limits_{t\in\{\old,\young\}}\sum\limits_{i=1}^{n_t} h(A_{ti}) \Big)^\eta \Big]\nn\\
	&\qquad\leq \E\Big[ \sum\limits_{t\in\{\old,\young\}}\sum\limits_{i=1}^{n_t} h(A_{ti})  \Big]^{1+\eta} ~.}
	\paragraph{Back to \eqref{for:lem:finite:higher-moment:1}.}
	By \eqref{for:lem:finite:higher-moment:1}, \eqref{for:lem:finite:higher-moment:13} and \eqref{for:lem:finite:higher-moment:14}, $\E\big[ \big(M^{(1)}_b(x,s)\big)^{1+\eta} \big]$ can be upper bounded by
	\eqan{\label{for:lem:finite:higher-moment:15}
		\E\big[ \big(M^{(1)}_b(x,s)\big)^{1+\eta} \big]
		\leq& \rho_b^{-(1+\eta)}\Big[2\E\Big[ \sum\limits_{t\in\{\old,\young\}}\sum\limits_{i=1}^{n_t} \Big(\frac{h(A_{ti})}{h(x,s)}\Big)^{1+\eta}\Big]\\
		&\hspace{1.75cm} + 2^\eta (b^{1-\chi}-1)^{2\eta} \E\Big[ \sum\limits_{k\in\{\old,\young\}}\sum\limits_{i=1}^{n_k} \frac{h(A_{ki})}{h(x,s)} \Big]^{1+\eta}  \Big]~.\nn}
	Note that $$\E\Big[ \sum\limits_{k\in\{\old,\young\}}\sum\limits_{i=1}^{n_k} \frac{h(A_{ki})}{h(x,s)} \Big]=\rho_b~.$$ Therefore, we are left with upper bounding the first expectation in the RHS of \eqref{for:lem:finite:higher-moment:15}. We calculate the expectation in RHS of \eqref{for:lem:finite:higher-moment:14} explicitly as
	\eqan{\label{for:lem:finite:higher-moment:002}
		&\rho_b^{-(1+\eta)}\E\Big[ \sum\limits_{t\in\{\old,\young\}}^\infty \sum\limits_{i=1}^{n_t} \Big(\frac{h(A_{ti})}{h(x,s)}\Big)^{1+\eta}\Big]\\
		=&\rho_b^{-(1+\eta)}\sum\limits_{t\in\{ \old,\young \}}\int\limits_{0}^\infty \pi\kappa_b((x,s),(y,t))\frac{h(y,t)^{1+\eta}}{h(x,s)^{1+\eta}}\,dy\nn\\
		=&\nu(\Old,\Young)\Big[\frac{c_{s\old}\bfu_{\old}^b}{\bfu_s^b}x^{(1+\eta)/2-\chi}\int\limits_{0}^x  y^{\chi-3/2-\eta/2} \,dy+\frac{c_{s\young}\bfu_{\young}^b}{\bfu_s^b}x^{\chi-1/2+\eta/2}\int\limits_{x}^{bx}  y^{-\chi-(1+\eta)/2}\,dy\Big]\nn\\
		=& \nu(\Old,\Young)\Big[\frac{c_{s\old}\bfu_{\old}^b}{\bfu_s^b} \frac{1}{\chi-1/2-\eta/2}+\frac{c_{s\young}(1-b^{1/2-\chi})\bfu_{\young}^b}{\bfu_s^b}\frac{1}{\chi-1/2+\eta/2}\Big]~, \nn
	}
	when $\eta<2\chi-1$ and $$\nu(\Old,\Young)=\pi\rho_b^{-(1+\eta)}\max\{ \big(\frac{\bfu_t^b}{\bfu_s^b}\big)^{\eta}:s,t\in\{\Old,\Young\}\}~.$$ Note that the RHS of \eqref{for:lem:finite:higher-moment:002} does not depend on $(x,s)$. Hence 
	$\E\big[\big(M^{(1)}_b(x,s)\big)^{1+\eta}\big]$ is finite and its upper bound is independent of $(x,s)$.
\end{proof}
Lemma~\ref{lem:subcriticality} and Proposition~\ref{prop:supercriticality} together prove that $1/r(\bar{\bfT}_{\kappa})$ is the critical percolation threshold for the P\'olya point tree, and its survival function is continuous at $1/r(\bar{\bfT}_{\kappa})$.
\section{Spectral norm of mean offspring operator of the P\'olya point tree}\label{sec:spectral_radius}
In \cite[Theorem~1.4]{RRR22}, it was proved that a wide class of preferential attachment models, including the models of concern in this paper, converges locally to the P\'olya point tree described in \cite{BergerBorgs}. The P\'olya point tree is a multi-type branching process with mixed discrete and continuous type space. 
In this section, we investigate some spectral properties of the mean offspring operator of the P\'olya point tree $\bfT_{\kappa}$. As described in Section~\ref{sec:proof-strategy}, instead of $\bfT_{\kappa}$, we work with $\bar{\bfT}_\kappa$, but in the end, we prove that the spectral and operator norm of both $\bfT_{\kappa}$ and $\bar{\bfT}_{\kappa}$ are equal. The only difference between the two is that we identify the explicit eigenfunction for $\bar{\bfT}_{\kappa}$ corresponding to $r(\bar{\bfT}_{\kappa})$, whereas we do not have the same for $\bfT_{\kappa}$. A very interesting property of these integral operators are that despite being non-self-adjoint, their spectral and operator norms are equal.

First, we demonstrate that the operator norm of the integral operator is bounded above for $\delta>0$. Subsequently, we establish the existence of a positive eigenfunction of the integral operator corresponding to this upper bound. Utilizing Gelfand's formula, we further prove that the spectral norm of a bounded integral operator is upper bounded by its operator norm. This allows us to explicitly determine the spectral norm along with its associated eigenfunction.

In summary, we present a comprehensive analysis of $\bar{\bfT}_\kappa$, highlighting its boundedness, eigenvalue existence and spectral norm determination for $\delta>0$. 
\begin{lemma}[Upper bound of the operator norm]\label{lem:upperbound}
	For $\delta>0,$ the operator norm of $\bar{\bfT}_{\kappa}$ and $\bfT_{\kappa}$ is upper bounded by 
	RHS of \eqref{nu-equality-PAM}.
\end{lemma}
\begin{proof}
	We use the Schur test \cite[Theorem (Schur's test, specific version)]{HV21} to upper bound the operator norm. Let $f((x,s))=\frac{\bfp_s}{\sqrt{x}}$ and $g((x,s))=\frac{\bfq_s}{\sqrt{x}},$ where $\bfp$ and $\bfq$ are to be chosen appropriately later on.
	Note that by \eqref{eq:kernel:offspring-operator:PPT},
	\eqan{\label{eq:schur:1}
		\sum\limits_{t\in\{ {\old,\young} \}}\int\limits_0^\infty \kappa((x,s),(y,t))f((y,t))\,dy=& c_{s\old}\bfp_{\old}\int\limits_0^x \frac{1}{x^{\chi}y^{1-\chi}}\frac{\,dy}{\sqrt{y}}+ c_{s\young}\bfp_{\young}\int\limits_x^\infty \frac{1}{x^{1-\chi}y^{\chi}}\frac{\,dy}{\sqrt{y}}\nn\\
		=&\frac{2}{2\chi-1}[c_{s\old}\bfp_{\old}+c_{s\young}\bfp_{\young}]\frac{1}{\sqrt{x}}~.}
	In the last equality, we use the fact that $\chi>1/2$ to compute the integral value. Similarly,
	\eqan{\label{eq:schur:2}
		&\sum\limits_{s\in\{ {\old,\young} \}}\int\limits_0^\infty \kappa((x,s),(y,t))g((x,s))\,dx\nn\\
		&\qquad=\sum\limits_{s\in\{ {\old,\young} \}} \Big[ c_{s\old}q_s\one_{\{t=\old\}}\int\limits_{y}^\infty \frac{1}{x^\chi y^{1-\chi}}\frac{\,dx}{\sqrt{x}}+c_{s\young}q_s\one_{\{t=\young\}}\int\limits_{0}^y \frac{1}{x^{1-\chi} y^{\chi}}\frac{\,dx}{\sqrt{x}} \Big]\\
		&\qquad= \frac{2}{2\chi-1}[\bfq_{\old} c_{\old t}+\bfq_{\young} c_{\young t}]\frac{1}{\sqrt{y}}~.\nn}
	Let us define the matrix $\bfM=(M_{s,t})_{s,t\in \{\old,\young\}}$ by
	\eqn{
		\label{def-bfM-PAM}
		\bfM
		=\left(\begin{matrix}
			c_{\old\old} &c_{\old\young}\\
			c_{\young\old} &c_{\young\young}
		\end{matrix}\right)
		.\nn}
	Therefore, \eqref{eq:schur:1} and \eqref{eq:schur:2} simplifies as
	\eqan{\label{eq:schur:3}
		\sum\limits_{t\in\{ {\old,\young} \}}\int\limits_0^\infty \kappa((x,s),(y,t))f((y,t))\,dy=& \frac{2}{2\chi-1} (\bfM\bfp)_s\frac{1}{\sqrt{x}},\\
		\label{eq:schur:4}
		\text{and}\qquad \sum\limits_{s\in\{ {\old,\young} \}}\int\limits_0^\infty \kappa((x,s),(y,t))g((x,s))\,dx=&\frac{2}{2\chi-1} (\bfM^*\bfq)_t\frac{1}{\sqrt{y}}~,} 
	where $\bfM^*$ is the transpose of $\bfM$. Thus, for the Schur-test in \cite[Theorem~(Schur’s test, specific version)]{HV21} with $p=p'=2$ to apply, we wish that 
	\eqn{ \label{eq:upperbound:requirement}
		\bfM \bfp=\lambda_{\sss\bfM}\bfq,
		\qquad{\text{and}}
		\qquad
		\bfM^* \bfq=\lambda_{\sss\bfM}\bfp.
	}
	Since $c_{\old\old}=c_{\young\young}$, the largest eigenvalue of $\bfM$ is given by
	\eqan{	\lambda_{\sss\bfM}\equiv c_{\old\old}+\sqrt{c_{\old\young}c_{\young\old}}=\frac{m(m+\delta)+\sqrt{m(m-1)(m+\delta)(m+1+\delta)}}{m+\delta} .\nonumber
	}
	Choosing $\bfp$ and $\bfq$ to be the right eigenvectors of $\bfM^*\bfM$ and $\bfM\bfM^*$, respectively, corresponding to the eigenvalue $\lambda_{\sss\bfM}^2$, we conclude that \eqref{eq:upperbound:requirement} holds.
	Hence, Lemma~\ref{lem:upperbound} follows immediately. Exactly the same calculation, adapted in the settings of $\bfT_{\kappa}$, proves that the operator norm of $\bfT_{\kappa}$ is also bounded above by the same.
\end{proof}

\begin{Remark}[Operator norm of ${\bfT}_{\kappa}$]\label{remark:operator-norm:restricted-space}
	In the restricted type space $\Scal$, it can also be shown that ${\bfT}_{\kappa}$ is a bounded operator from $L^2(\Scal,\lambda)$ to itself.
	The upper bound on the operator norm is already obtained in Lemma~\ref{lem:upperbound}.  To furnish a lower bound, we produce a series of $L^2(\Scal,\lambda)$-normalized functions $f_\vep(x,s)=\bfp_s/\sqrt{x\log(1/\vep)}\one_{\{x\geq \vep\}}$ and then use that $\| {\bfT}_{\kappa} \|\geq \langle {\bfT}_{\kappa} f_\vep,{\bfT}_{\kappa} f_\vep \rangle$. By allowing $\vep$ to approach zero, we witness the convergence of $\langle {\bfT}_{\kappa} f_\vep,{\bfT}_{\kappa} f_\vep \rangle$ towards the square of the RHS of \eqref{nu-equality-PAM},  
	proving it to be the $L^2(\Scal,\lambda)$-operator norm of $\bar{\bfT}_{\kappa}$.
	\hfill$\blacksquare$
\end{Remark}
\begin{proof}[Proof of Theorem~\ref{thm:operator-norm:PPT}]
	We show that there exists an eigenfunction of $\bar{\bfT}_\kappa$ with the eigenvalue $r(\bar{\bfT}_\kappa)$. If we choose $\bfp$ to be the right eigenvector of $\bfM$ with eigenvalue $\lambda_{\sss \bfM}$, then \eqref{eq:schur:3} simplifies to
	\eqn{\label{eq:eigenfunction:1}
		\sum\limits_{t\in\{ {\old,\young} \}}\int\limits_0^\infty \kappa((x,s),(y,t))f((y,t))\,dy= \frac{2}{2\chi-1} \lambda_{\sss\bfM}\bfp_s\frac{1}{\sqrt{x}}=r(\bar{\bfT}_\kappa)f((x,s))~.}
	The operator $\bar{\bfT}_\kappa$ possesses a positive eigenfunction with eigenvalue $r(\bar{\bfT}_\kappa)$. As $\bar{\bfT}_\kappa$ is a bounded linear operator, we can employ Gelfand's theorem to obtain that the operator norm of $\bar{\bfT}_\kappa$ is at least its spectral norm. Since $r(\bar{\bfT}_{\kappa})$ is an eigenvalue of $\bar{\bfT}_{\kappa}$, the spectral norm of $\bar{\bfT}_{\kappa}$ as well as the operator norm is lower bounded by $r(\bar{\bfT}_{\kappa})$. Moreover, Lemma~\ref{lem:upperbound} proves that the operator norm is also bounded above by the same number. Consequently, we establish that $r(\bar{\bfT}_\kappa)$ indeed represents both the operator norm and spectral norm of $\bar{\bfT}_\kappa$ in the extended type space $\Scal_e=[0,\infty)\times\{\Old,\Young\}$.
	
	Although we cannot explicitly provide the eigenfunction corresponding to $r(\bar{\bfT}_{\kappa})$, we prove that the spectral norm of ${\bfT}_{\kappa}$ is also $r(\bar{\bfT}_{\kappa})$, following an argument involving the numerical radius of ${\bfT}_{\kappa},~w({\bfT}_{\kappa})$. The numerical radius of any bounded operator is upper bounded by its operator norm. With the sequence $f_\vep$ defined in Remark~\ref{remark:operator-norm:restricted-space} and taking $\vep$ to $0$, we obtain $r(\bar{\bfT}_{\kappa})$ as a lower bound for $w({\bfT}_{\kappa})$ as well, proving $w({\bfT}_{\kappa})=\|{\bfT}_{\kappa}\|$. Since ${\bfT}_{\kappa}$ is a bounded operator on a Hilbert space $L^2(\Scal,\lambda),$ the previous equality implies that the spectral norm and numerical radius of ${\bfT}_{\kappa}$ are equal. 
\end{proof}

\begin{Remark}[Equality of operator norm and spectral norm]
	Although $\bar{\bfT}_{\kappa}$ and $\bfT_{\kappa}$ are not self-adjoint operators, their spectral norm and operator norm are equal.
\end{Remark}
Following similar steps we now prove Lemma~\ref{lem:spectral-norm:truncated-operator}.
\begin{proof}[Proof of Lemma~\ref{lem:spectral-norm:truncated-operator}]
	We proceed similarly as we did in Theorem~\ref{thm:operator-norm:PPT}. It can easily be shown that by the Schur test we can upper bound the operator norm by $r(\bar{\bfT}_{\kappa_b})$. Now we show that $h(x,s)$ in \eqref{eq:lemma:spectral-radius:efunc}, is the eigenfunction of $\bar{\bfT}_{\kappa_b}$ with eigenvalue $r(\bar{\bfT}_{\kappa_b})$. Therefore, for any $(x,s)\in\Scal_e$,
	\eqan{\label{for:lem:spectra:truncated-operator:1}
		\sum\limits_{t\in\{\old,\young\}} \int\limits_0^\infty \kappa_b((x,s),(y,t))h(y,s)\,dy
		=&\sum\limits_{t\in\{\old,\young\}} \int\limits_0^{bx} \kappa((x,s),(y,t))h(y,s)\,dy\nn\\
		=&\frac{x^{-1/2}}{\chi-\frac{1}{2}}\big[ c_{so}\bfu_{\old}^b +c_{s\young}(1-b^{1/2-\chi})\bfu_{\young}^b\big]~.
	}
	Since $\bfu^b$ is the eigenvector of $\bfM_b$ corresponding to its largest eigenvalue $\lambda_{\sss \bfM}^{\sss (b)}$ given by	\eqn{\label{for:lem:spectra:truncated-operator:2}
		\lambda_{\sss \bfM}^{\sss (b)} = \frac{(1+q)c_{\old\old}+\sqrt{((1-q)c_{\old\old})^2+4qc_{\old\young}c_{\young\old}}}{2}~,}
	where $q=1-b^{1/2-\chi}$, \eqref{for:lem:spectra:truncated-operator:1} leads to
	\eqn{\label{for:lem:spectra:truncated-operator:3}
		\Big(\bar{\bfT}_{\kappa_b}h\Big)(x,s)=\frac{\lambda_{\sss \bfM}^{\sss (b)}}{2\chi-1}\frac{\bfu_s^b}{\sqrt{x}}=r(\bar{\bfT}_{\kappa_b})h(x,s)~.}
	Hence $h$ is the eigenfunction of the truncated integral operator $\bar{\bfT}_{\kappa_b}$ corresponding to its spectral norm $r(\bar{\bfT}_{\kappa_b})$ in \eqref{eq:lemma:spectral-radius}.
\end{proof}




\backmatter


\begin{wrapfigure}{r}{0.1\textwidth}
	\includegraphics[width=\linewidth]{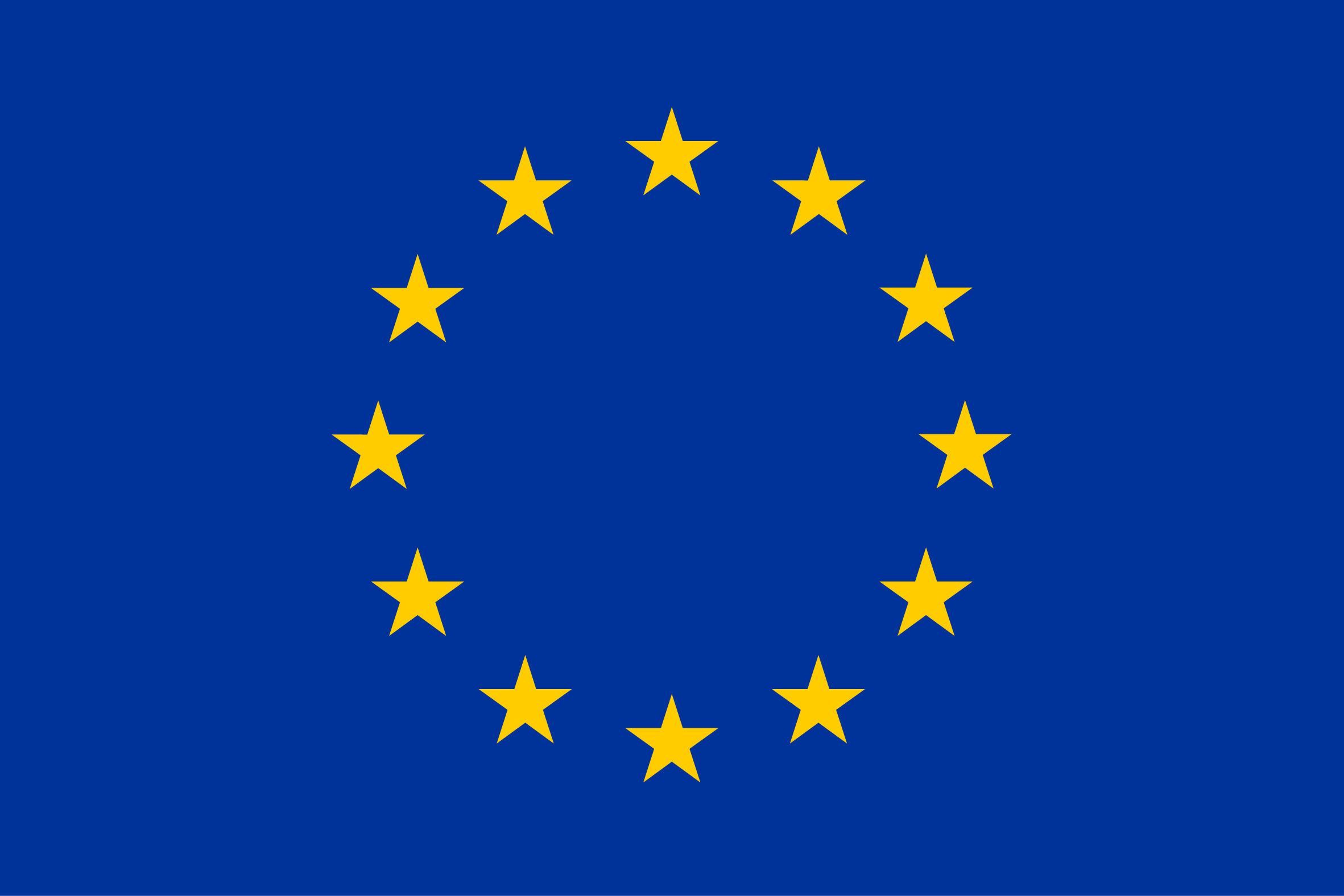} 
	\label{fig:wrapfig}
\end{wrapfigure}
\bmhead{Acknowledgments}
This work is supported in part by the Netherlands Organisation for Scientific Research (NWO) through the Gravitation {\sc Networks} grant 024.002.003. The work of RR is further supported by the European Union's Horizon 2020 research and innovation programme under the Marie Sk\l{}odowska-Curie grant agreement no.\ 945045.






\begin{appendices}
\appendix
\section{Expansion properties of preferential attachment models}
The proof of Lemma~\ref{lem:intermediate:Cheeger-value} follows by adapting the proof of \cite[Lemma~2]{MPS06} for the case of model (a), (b) and (d) with general initial graph settings and any admissible $\delta$. 
\begin{proof}[Proof of Lemma~\ref{lem:intermediate:Cheeger-value}]
	Since models (a),(b) and (d) can all be formed using collapsing procedure mentioned in \cite{vdH1}, we prove the next few steps for all of these models. Recall the definitions of \rm{BAD} set from \eqref{def:bad-set}, and \rm{ILL} and \rm{FIT} mini-vertices from Definition~\ref{def:fit-ill-minivertex}.
	
	Let $A_1\subset A$ such that the mini-vertices in $A_1$ are associated with vertices in $S,$ and $A_2=A\setminus A_1$ such that the mini-vertices in $A_2$ are associated with vertices in $S^c$. Denote $|A_1|=k_1$ and $|A_2|=k_2$ with $|A|=k_1+k_2$. Now there are $mk-k_1$ many mini-vertices that are associated with vertices in $S,$ but are not in $A$. Similarly, there are $m_{[n]}-mk-k_2$ many mini-vertices that are associated with vertices in $S^c,$ but are not in $A$. Let $x_1<x_2<\cdots<x_{mk-k_1}$ be the mini-vertices that are associated with vertices in $S,$ but not in $A$, and $\Bar{x}_1<\Bar{x}_2<\cdots<\Bar{x}_{m_{[n]}-mk-k_2}$ be the mini-vertices that are associated with vertices in $S^c,$ but not in $A$. For $i\in[mk-k_1],$
	\eqn{\label{for:lem:intermediate:Cheeger-value:1}
		x_i=y_i+z_i+1~,}
	where $y_i$ is the number of mini-vertices from $A$, that arrived prior to $x_i$, and $z_i$ is the number of mini-vertices from $[m_{[n]}]\setminus A$, that arrived prior to $x_i$. Similarly, for $i\in[m_{[n]}-mk-k_2],$
	\eqn{\label{for:lem:intermediate:Cheeger-value:2}
		\Bar{x}_i=\Bar{y}_i+\Bar{z}_i+1~,}
	where $\Bar{y}_i$ is the number of mini-vertices from $A$, that arrived prior to $\Bar{x}_i$, and $\Bar{z}_i$ is the number of mini-vertices from $[m_{[n]}]\setminus A$, that arrived prior to $\Bar{x}_i$.
	Note that $z_i$ or $\Bar{z}_i$ counts the total number of $x_j$ and $\Bar{x}_j$ appeared before $x_i$ or $\Bar{x}_i$, respectively. If we order $\{ x_i \}_{i\in[mk-k_1]}\cup \{ \Bar{x}_i\}_{i\in[m_{[n]}-mk-k_2]}$ in an increasing order, then the corresponding $z_i$ or $\Bar{z}_i$ values increase by $1$ every time, and these values ranges from $0$ to $m_{[n]}-k_1-k_2$. Therefore,
	\eqn{\label{for:lem:intermediate:Cheeger-value:3}
		\bigcup\limits_{i=1}^{mk-k_1}\{z_i+1\}\bigcup \bigcup\limits_{i=1}^{m_{[n]}-mk-k_2}\{\Bar{z}_i+1\}=[m_{[n]}-k_1-k_2]
		=[m_{[n]}-|A|]~.}
	The total volume of the tree when $t$-th mini-vertex join the tree is $2(t-1)$. If $t=x_i$ (or $\Bar{x}_i$), then the total volume equals
	\eqn{\label{for:lem:intermediate:Cheeger-value:4}
		2z_i+2y_i\quad \text{or}\quad  2\Bar{z}_i+2\Bar{y}_i~.}
	Now we have the following cases:
	\begin{enumerate}
		\item \textbf{Case 1:} Both $1$ and $2$ are in the same set. In this scenario, we consider $S$ to be the one containing both the vertices.
		\item \textbf{Case 2:} $1$ and $2$ are in different sets. In this scenario, we consider $S$ to be the set containing $1$. Note that $a_2\leq m$.
	\end{enumerate}
	Since we have finitely many (and deterministically depending only on the structure of the initial graph) mini-vertices with deterministic connection, we can safely assume that the upper bound which we are going to prove works for every \rm{ILL} mini-vertex except for first few. Since we aim to prove an exponential bound in the end, these finitely many disturbances make no difference in the end result. So from now on, we keep on upper bounding the non-deterministic edges, and assume that the upper bound holds for the deterministic edges also.

	Note that we are considering the case where all mini-vertices in $A$ are FIT, whereas all mini-vertices in $[m_{[n]}]\setminus{A}$ are ILL. When $x_i$ arrives, the total volume of the mini-vertices associated to the vertices of $S$ consists of
	\begin{enumerate}
		\item[(a)] all ILL mini-vertices that arrived prior to $x_i$, and are associated to a vertex in $S$. Each such mini-vertex contributes $2$ to the total volume of the mini-vertices associated to the vertices of $S$. There are $(i-1)$ such mini-vertices contributing a total of $2(i-1)$ to the sum;
		\item[(b)] all FIT mini-vertices that arrived prior to $x_i$, and are associated to a vertex in $S$. Each such mini-vertex contributes $1$ to the total volume of the mini-vertices associated to the vertices of $S$. There are $y_i$ such mini-vertices contributing a total of $y_i$ to the sum.
	\end{enumerate}
	Therefore when $x_i$ arrives, the volume of the mini-vertices associated to the vertices of $S$ is $2(i-1)+y_i$. Now we move on to bound the probability of creating \rm{ILL} mini-vertices conditionally on all mini-vertices in $A$ being \rm{FIT} for models (a), (b) and (d). First we start with model (b) and then adapt the bound for models (a) and (d) respectively. Recall that $a_{[2]}=a_1+a_2$ is the total degree of the initial graph.
	\paragraph{Model (b).}
	The probability that $x_i$ connects to mini-vertices associated with vertices in $S$, and becomes \rm{ILL} given that all mini-vertices in $A$ which arrived prior to $x_i$ are \rm{FIT}, and both $1$ and $2$ are in $S$, is 
	\eqn{\label{for:lem:intermediate:Cheeger-value:5}
		\frac{2(i-1)+y_i+(i-1+y_i-a_{[2]})\delta/m+2\delta}{2(y_i+z_i)+(y_i+z_i-a_{[2]})\delta/m+2\delta }~.}
	For any $\delta>-m$ and $y_i\geq 0$, the probability can be upper bounded by
	\eqan{\label{for:lem:intermediate:Cheegar-value:06}
		&\frac{2(i-1)+y_i+(i-1+y_i)\delta/m+\delta(2-a_{[2]}/m)}{2(y_i+z_i)+(y_i+z_i)\delta/m+\delta(2-a_{[2]}/m)}\nn\\
		\leq& \frac{2(i-1+y_i)+(i-1+y_i)\delta/m+\delta(2-a_{[2]}/m)}{2(y_i+z_i)+(y_i+z_i)\delta/m+\delta(2-a_{[2]}/m) }~. }
	Since $2+\delta/m> 1$ and $|y_i|\leq |A|$, we further upper bound the probability in \eqref{for:lem:intermediate:Cheeger-value:5} by
	\eqan{\label{for:lem:intermediate:Cheeger-value:6}
		&\frac{i+|A|-1+\delta(2-a_{[2]}/m)/(2+\delta/m)}{z_i+|A|+\delta(2-a_{[2]}/m)/(2+\delta/m)}\nn\\
		\leq & \frac{i+|A|+\delta(2-a_{[2]}/m)/(2+\delta/m)}{z_i+|A|+1+\delta(2-a_{[2]}/m)/(2+\delta/m)}~.
	}
	Instead, for the case when $1\in S$, but $2\in S^c$, the edge-connection probability in \eqref{for:lem:intermediate:Cheeger-value:5} changes to
	\eqn{\label{for:lem:intermediate:Cheegar-value:08}
		\frac{2(i-1)+y_i+(i-1+y_i-a_{1})\delta/m+\delta}{2(y_i+z_i)+(y_i+z_i-a_{[2]})\delta/m+2\delta }~.}
	Since $a_2\leq m$, for $\delta\geq 0$ we bound the edge-connection probability by
	\eqn{\label{for:lem:intermediate:Cheegar-value:09}
		\frac{2(i-1)+y_i+(i-1+y_i-a_{1})\delta/m+\delta(1-a_{1}/m)}{2(y_i+z_i)+(y_i+z_i-a_{[2]})\delta/m+\delta(1-a_{1}/m) }~,}
	whereas, for $\delta\in(-m,0)$, the probability is upper bounded by
	\eqn{\label{for:lem:intermediate:Cheegar-value:10}
		\frac{2(i-1)+y_i+(i-1+y_i-a_{1})\delta/m+\delta(2-a_{[2]}/m)}{2(y_i+z_i)+(y_i+z_i-a_{[2]})\delta/m+\delta(2-a_{[2]}/m) }~.}
	Following a similar upper bound in \eqref{for:lem:intermediate:Cheegar-value:06} and \eqref{for:lem:intermediate:Cheeger-value:6}, we upper bound \eqref{for:lem:intermediate:Cheegar-value:09} by
	\eqn{\label{for:lem:intermediate:Cheegar-value:11}
		\frac{i+|A|-1+\delta(1-a_{1}/m)/(2+\delta/m)}{z_i+|A|+\delta(1-a_{1}/m)/(2+\delta/m)}~,}
	and upper bound \eqref{for:lem:intermediate:Cheegar-value:11} by
	\eqan{\label{for:lem:intermediate:Cheegar-value:12}
		&\frac{i+|A|-1+\delta(2-a_{[2]}/m)/(2+\delta/m)}{z_i+|A|+\delta(2-a_{[2]}/m)/(2+\delta/m)}\nn\\
		\leq&\frac{i+|A|+\delta(2-a_{[2]}/m)/(2+\delta/m)}{z_i+|A|+1+\delta(2-a_{[2]}/m)/(2+\delta/m)}~.}
	Denote $\delta_0=\max\{0,\lceil\delta(2-a_{[2]}/m)/(2+\delta/m)\rceil,\lceil\delta(1-a_{1}/m)/(2+\delta/m)\rceil\}$. Then we bound the connection probability in \eqref{for:lem:intermediate:Cheeger-value:5} and \eqref{for:lem:intermediate:Cheegar-value:09} by
	\eqn{\label{for:lem:intermediate:Cheegar-value:13} \frac{i+|A|+\delta_0}{z_i+|A|+1+\delta_0}~.}
	\paragraph{Adaptation to model (a).}
	In model (a), the total volume of the graph is always one more than that in model (b). For model (a), $x_i$ can become \rm{ILL} either by creating a self-loop or connecting to the mini-vertices in $S$ and hence the probability in \eqref{for:lem:intermediate:Cheeger-value:5} becomes
	\eqn{\label{for:lem:intermediate:Cheeger-value:7}
		\frac{2(i-1)+y_i+1+(i+y_i-a_{[2]})\delta/m+2\delta}{2(y_i+z_i)+1+(y_i+z_i+1-a_{[2]})\delta/m+2\delta }~,}
	which can be bounded by
	\eqn{\label{for:lem:intermediate:Cheeger-value:7-1}
	\frac{2i+y_i+(i+y_i-a_{[2]})\delta/m+2\delta}{2(y_i+z_i+1)+(y_i+z_i+1-a_{[2]})\delta/m+2\delta }~.}
	Following exactly the same bounding technique in \eqref{for:lem:intermediate:Cheegar-value:06}-\eqref{for:lem:intermediate:Cheeger-value:6}, we upper bound \eqref{for:lem:intermediate:Cheeger-value:7-1} by the RHS of \eqref{for:lem:intermediate:Cheeger-value:6}.
	Similarly \eqref{for:lem:intermediate:Cheegar-value:08} is reformulated as
	\eqn{\label{for:lem:intermediate:Cheeger-value:14}
		\frac{2(i-1)+y_i+1+(i+y_i-a_{1})\delta/m+\delta}{2(y_i+z_i)+1+(y_i+z_i+1-a_{[2]})\delta/m+2\delta }~, }
	which can be bounded by
	\eqn{\label{for:lem:intermediate:Cheeger-value:14-1}
	\frac{2i+y_i+(i+y_i-a_{1})\delta/m+\delta}{2(y_i+z_i+1)+(y_i+z_i+1-a_{[2]})\delta/m+2\delta }~.}
	We follow \eqref{for:lem:intermediate:Cheegar-value:09}-\eqref{for:lem:intermediate:Cheegar-value:12} to further upper bound \eqref{for:lem:intermediate:Cheeger-value:14-1} by RHS of \eqref{for:lem:intermediate:Cheegar-value:12}. Therefore with the same $\delta_0$ chosen earlier for model (b), we upper bound the probabilities in \eqref{for:lem:intermediate:Cheeger-value:7} and \eqref{for:lem:intermediate:Cheeger-value:14} by \eqref{for:lem:intermediate:Cheegar-value:13}.
	\paragraph{Adaptation to model(d).}
	In model (d), the difference is that the mini-vertices cannot attach to any mini-vertex associated with the same vertex. To tackle this difficulty we associate a mark to every mini-vertex counting the number of mini-vertices that appeared before it, and are associated to the same vertex in the graph. Define $\{x_i,y_i,z_i\}$ and $\{\bar{x}_i,\bar{y}_i,\bar{z}_i\}$ similarly as before. Note that every mini-vertex $x_i$ (and $\bar{x}_i$) has a score $r_i$ (and $\bar{r}_i$) taking values in $[m]$, representing the edge the mini-vertex creates from the vertex it is associated to. In fact every mini-vertex has such a score, but we refrain from defining that as it complicates the notation. Note that in model (d), every mini-vertex $x_i$ and $\bar{x}_i$ can attach to any mini-vertex except for $r_i$ many mini-vertices (including itself). 
	
	When $x_i$ (or $\bar{x}_i$) joins the graph, the total volume of the graph to which it can attach to is
	\eqn{\label{for:lem:model-d:1}
		2(y_i+z_i)-r_i\qquad\text{or}\qquad 2(\bar{y}_i+\bar{z}_i)-\bar{r}_i~.}
	When $x_i$ arrives, the total volume of $S$ to which $x_i$ may connect to is similarly calculated.
	For the case $1,2\in S$, the probability that $x_i$ becomes \rm{ILL} is
	\eqan{\label{for:lem:model-d:2}
		&\frac{2(i-1)+y_i-r_i+(i-1+y_i-r_i-a_{[2]})\delta/m+2\delta}{2(y_i+z_i)-r_i+(y_i+z_i-r_i-a_{[2]})\delta/m+2\delta}\nn\\
		&\qquad=\frac{2(i-1)+y_i-r_i(1+\delta/m)+(i-1+y_i-a_{[2]})\delta/m+2\delta}{2(y_i+z_i)-r_i(1+\delta/m)+(y_i+z_i-a_{[2]})\delta/m+2\delta}~,}
	whereas for the case $1\in S,$ but $2\notin S$, the probability is 
	\eqan{\label{for:lem:model-d:3}
		&\frac{2(i-1)+y_i-r_i+(i-1+y_i-r_i-a_{1})\delta/m+\delta}{2(y_i+z_i)-r_i+(y_i+z_i-r_i-a_{[2]})\delta/m+2\delta}\nn\\
		&\qquad=\frac{2(i-1)+y_i-r_i(1+\delta/m)+(i-1+y_i-a_{1})\delta/m+\delta}{2(y_i+z_i)-r_i(1+\delta/m)+(y_i+z_i-a_{[2]})\delta/m+2\delta}~.}
	Note that here we have to subtract a factor of $r_i$ in the denominator also to compensate for the fact that $x_i$ cannot connect to $r_i$ many mini-vertices to become ILL. 
	Since $\delta>-m$, we can upper bound the probability in \eqref{for:lem:model-d:2} and \eqref{for:lem:model-d:3} by \eqref{for:lem:intermediate:Cheeger-value:5} and \eqref{for:lem:intermediate:Cheegar-value:08}, respectively. 
	\paragraph{Conclusion of the proof for all models.}
	Similarly we calculate the probability that $\Bar{x}_i$ becomes \rm{ILL}, given that all mini-vertices that arrived before $\Bar{x}_i$ and belong to $A$ are \rm{FIT}, while those belonging to $A^c$ are \rm{ILL} for models (a), (b) and (d) of preferential attachment models, and we can bound these probabilities by $$\frac{i+|A|+\delta_0}{\Bar{z}_i+|A|+1+\delta_0}~.$$
	The probability that all mini-vertices associated with $[m_{[n]}]\setminus A$ are ILL, given that all mini-vertices in $A$ are FIT, is at most
	\eqn{\label{for:lem:intermediate:Cheeger-value:9}
		\Big( \prod\limits_{i=1}^{mk-k_1}\frac{i+|A|+\delta_0}{z_i+1+|A|+\delta_0} \Big)\Big( \prod\limits_{i=1}^{m_{[n]}-mk-k_2}\frac{i+|A|+\delta_0}{\Bar{z}_i+1+|A|+\delta_0} \Big)
		~.}
	Using the relation obtained in \eqref{for:lem:intermediate:Cheeger-value:3}, we simplify \eqref{for:lem:intermediate:Cheeger-value:9} as
	\eqan{\label{for:lem:intermediate:Cheegar-value:15}
		&\Big( \prod\limits_{i=1}^{mk-k_1}\frac{i+|A|+\delta_0}{z_i+1+|A|+\delta_0} \Big)\Big( \prod\limits_{i=1}^{m_{[n]}-mk-k_2}\frac{i+|A|+\delta_0}{\Bar{z}_i+1+|A|+\delta_0} \Big)\nn\\
		&\qquad\qquad=\frac{\prod\limits_{i=1}^{mk-k_1}(i+|A|+\delta_0)\prod\limits_{i=1}^{m_{[n]}-mk-k_2}(i+|A|+\delta_0)}{\prod\limits_{i=1}^{m_{[n]}-|A|}(i+|A|+\delta_0)}~.
	}
	Now following the similar calculation in \cite[equations (8)-(10)]{MPS06}, the RHS of \eqref{for:lem:intermediate:Cheegar-value:15} is further simplified and upper bounded as
	\eqan{\label{for:lem:intermediate:Cheegar-value:16}
		\text{RHS of \eqref{for:lem:intermediate:Cheegar-value:15}}=&~\frac{\prod\limits_{i=1}^{mk-k_1+|A|+\delta_0}i\prod\limits_{i=1}^{m_{[n]}-mk-k_2+|A|+\delta_0}i}{\prod\limits_{i=1}^{m_{[n]}+\delta_0}i\prod\limits_{i=1}^{|A|+\delta_0} i}\nn\\
		=&~\frac{(mk+k_2+\delta_0)!(m_{[n]}-mk+k_1+\delta_0)!}{(m_{[n]}+\delta_0)! (|A|+\delta_0)!}\\
		=&~\prod\limits_{i=0}^{k_2-1}\Big( \frac{mk+k_2+\delta_0-i}{m_{[n]}+\delta_0-i} \Big) \prod\limits_{i=0}^{k_1-1}\Big( \frac{m_{[n]}-mk+k_1+\delta_0-i}{m_{[n]}-k_2+\delta_0-i} \Big)\nn\\
		&\hspace{5cm}\times\frac{(mk+\delta_0)!(m_{[n]}-mk+\delta_0)!}{(m_{[n]}-|A|+\delta_0)!(|A|+\delta_0)!}\nn\\
		\leq&~ \frac{(mk+\delta_0)!(m_{[n]}-mk+\delta_0)!}{(m_{[n]}-|A|+\delta_0)!(|A|+\delta_0)!}~,\nn
	}
	where the inequality follows from the fact that
	\eqan{\label{for:lem:intermediate:Cheegar-value:17}
		\frac{mk+k_2+\delta_0-i}{m_{[n]}+\delta_0-i}\leq&~ 1\qquad\text{for all }i\leq k_2-1~,\nn\\
		\mbox{and}\qquad\frac{m_{[n]}-mk+k_1+\delta_0-i}{m_{[n]}-k_2+\delta_0-i} \leq&~1\qquad\text{for all }i\leq k_1-1~.
	}
	Although the calculations of Proposition~\ref{prop:Cheeger-value:model-b} follow from the bound obtained in \eqref{for:lem:intermediate:Cheegar-value:17}, we upper bound it further to reduce a similar repetitive calculation as in \cite[Proof of Theorem~1]{MPS06} with minor modification. Since we aim to obtain exponential upper bound on the probability in Proposition~\ref{prop:Cheeger-value:model-b}, we use this leverage to further upper bound the probability as
	\eqan{\label{for:lem:intermediate:Cheegar-value:18}
		&\frac{(mk+\delta_0)!(m_{[n]}-mk+\delta_0)!}{(m_{[n]}-|A|+\delta_0)!(|A|+\delta_0)!}\nn\\
		=&~\prod\limits_{i=0}^{\delta_0-1}\Big( \frac{mk+\delta-i}{|A|+\delta_0-i} \Big)\prod\limits_{i=0}^{\delta_0-1}\Big( \frac{m_{[n]}-mk+\delta_0-i}{m_{[n]}-|A|+\delta_0-i} \Big)\frac{(mk)!(m_{[n]}-mk)!}{(m_{[n]}-|A|)!|A|!}\\
		\leq&~(mn+\delta_0)^{\delta_0}
		\binom{mk}{|A|}\binom{m_{[n]}-|A|}{mk-|A|}^{-1}~,\nn
	}
	where the inequality follows from the fact that $|A|\leq \alpha k<mk<mn$.
	Since $|A|\leq \alpha k$, and the upper bound is monotonically increasing in the size of $A$, the bound in the lemma follows immediately from \eqref{for:lem:intermediate:Cheegar-value:18}. 
\end{proof}
\section{Proof for the sufficient condition}\label{app:sufficiency}
In this section we provide the proof for the sufficiency condition in Lemma~\ref{lem:equivalent:condition}. The proof follows immediately from Markov's inequality.
\begin{proof}[Proof of Claim~\ref{lem:equivalent:condition}]
	First we simplify $f(\e^t)$ defined in \eqref{eq:K-S:07}. For any $t>1$ and $\eta>0$,
	\eqn{\label{lem:eqv:1}
	\bar{f}(\e^t)=\sup\limits_{x}\prob\Big( \frac{V(\tilde{\xi}^x)}{\rho_b h(x)}>\e^t \Big)=\sup\limits_{x}\prob\Big( \left(\frac{V(\tilde{\xi}^x)}{\rho_b h(x)}\right)^{\eta}>\e^{\eta t} \Big)~.}
	By Markov's inequality, we bound $\bar{f}(\e^t)$ by
	\eqn{\label{lem:eqv:2}
	\sup\limits_x \e^{-\eta t}\E\Big[ \left(\frac{V(\tilde{\xi}^x)}{\rho_b h(x)}\right)^{\eta} \Big]=\e^{-\eta t}\sup\limits_x \E\Big[ \left(\frac{V({\xi}^x)}{\rho_b h(x)}\right)^{1+\eta} \Big]~.}
	The equality follows from the fact that $\tilde{\xi}^x$ is size-biased version of $\xi^x$ as defined in \eqref{eq:K-S:2}. Since, $\sup\limits_x \E\Big[ \left(\frac{V({\xi}^x)}{\rho_b h(x)}\right)^{1+\eta} \Big]$ finite, we bound $\int\limits_1^\infty \bar{f}(\e^t)\,dt$ by
	\eqn{\label{eq:eqv:3}
	\int\limits_1^\infty \bar{f}(\e^t)\,dt \leq \sup\limits_x \E\Big[ \left(\frac{V({\xi}^x)}{\rho_b h(x)}\right)^{1+\eta} \Big] \e^{-\eta}/\eta<\infty~,}
which completes the proof.
\end{proof}




\end{appendices}

\bibliographystyle{acm}
\bibliography{bibliofile}

\end{document}